\documentclass[oneside,reqno,english]{amsart}
\usepackage[T1]{fontenc}
\usepackage[utf8]{inputenc}
\usepackage{xcolor}
\usepackage{babel}
\usepackage{prettyref}
\usepackage{amstext}
\usepackage{amsthm}
\usepackage{amssymb}
\usepackage[pdfusetitle,
 bookmarks=true,bookmarksnumbered=false,bookmarksopen=false,
 breaklinks=false,pdfborder={0 0 0},pdfborderstyle={},backref=false,colorlinks=false]
 {hyperref}
\hypersetup{
 colorlinks=true,citecolor=blue,linkcolor=blue,linktocpage=true}

\makeatletter
\numberwithin{equation}{section}
\numberwithin{figure}{section}

\usepackage{prettyref}

\newrefformat{cor}{Corollary~\ref{#1}}
\newrefformat{subsec}{Section~\ref{#1}}
\newrefformat{lem}{Lemma~\ref{#1}}
\newrefformat{thm}{Theorem~\ref{#1}}
\newrefformat{sec}{Section~\ref{#1}}
\newrefformat{chap}{Chapter~\ref{#1}}
\newrefformat{prop}{Proposition~\ref{#1}}
\newrefformat{exa}{Example~\ref{#1}}
\newrefformat{tab}{Table~\ref{#1}}
\newrefformat{rem}{Remark~\ref{#1}}
\newrefformat{def}{Definition~\ref{#1}}
\newrefformat{fig}{Figure~\ref{#1}}
\newrefformat{claim}{Claim~\ref{#1}}
\newrefformat{assu}{Assumption~\ref{#1}}

\usepackage{tikz}
\usetikzlibrary{trees,positioning}

\makeatother

\theoremstyle{plain}
\newtheorem{thm}{\protect\theoremname}[section]
\newtheorem{lem}[thm]{\protect\lemmaname}
\theoremstyle{remark}
\newtheorem{rem}[thm]{\protect\remarkname}
\theoremstyle{plain}
\newtheorem{cor}[thm]{\protect\corollaryname}
\newtheorem{prop}[thm]{\protect\propositionname}
\theoremstyle{definition}
\newtheorem{defn}[thm]{\protect\definitionname}
\theoremstyle{plain}
\newtheorem{assumption}[thm]{\protect\assumptionname}
\theoremstyle{remark}
\newtheorem*{rem*}{\protect\remarkname}
\newtheorem*{notation*}{\protect\notationname}
\providecommand{\assumptionname}{Assumption}
\providecommand{\corollaryname}{Corollary}
\providecommand{\definitionname}{Definition}
\providecommand{\lemmaname}{Lemma}
\providecommand{\notationname}{Notation}
\providecommand{\propositionname}{Proposition}
\providecommand{\remarkname}{Remark}
\providecommand{\theoremname}{Theorem}

\begin{document}
\title[]{Boundary Disintegration for Weighted Residual Energy Trees}
\begin{abstract}
We study iterated weighted residual (WR) splittings generated by a
positive operator $R_{0}\in B\left(H\right)_{+}$ and a finite family
of contractions $C_{1},\dots,C_{m}$ in $B\left(H\right)$. The associated
residual update $R\mapsto R^{1/2}(I-C^{*}_{j}C_{j})R^{1/2}$ produces
an $m$-ary energy tree of residuals $\left\{ R_{w}\right\} $ and
dissipated pieces $\left\{ D_{w,j}\right\} $ indexed by finite words.
From this tree we construct intrinsic path measures on the path space
by biasing transitions either by a fixed quadratic form $x\mapsto\left\langle x,D_{w,j}x\right\rangle $
(defining the measures $\nu_{x}$) or, in the trace-class setting,
by ${\rm tr}\left(D_{w,j}\right)$ (yielding a reference measure $\nu_{\mathrm{tr}}$).
When $R_{0}\in S_{1}\left(H\right)_{+}$, we show that $\nu_{\mathrm{tr}}$
dominates the family $\left\{ \nu_{x}\right\} $ and identify $d\nu_{x}/d\nu_{\mathrm{tr}}$
as a canonical martingale limit of cylinder likelihood ratios. Along
$\nu_{\mathrm{tr}}$-almost every branch the residuals decrease to
a terminal trace-class random variable $R_{\infty}$, which we interpret
as the WR boundary variable. We then disintegrate $\nu_{\mathrm{tr}}$
over $\sigma\left(R_{\infty}\right)$, obtaining a boundary law $\mu_{\mathrm{tr}}=\left(R_{\infty}\right)_{\#}\nu_{\mathrm{tr}}$
and conditional path measures $\left\{ \nu^{T}_{\mathrm{tr}}\right\} $.
Finally, we show that each $\nu_{x}$ admits a boundary representation
as a mixture of $\left\{ \nu^{T}_{\mathrm{tr}}\right\} $ with an
explicit boundary density $h_{x}=d\mu_{x}/d\mu_{\mathrm{tr}}$, thereby
organizing the family of intrinsic WR path measures by a single trace-biased
boundary disintegration.
\end{abstract}

\author{James Tian}
\address{Mathematical Reviews, 535 W. William St, Suite 210, Ann Arbor, MI
48103, USA}
\email{james.ftian@gmail.com}

\maketitle
\tableofcontents{}

\section{Introduction}\label{sec:1}

Let $H$ be a (complex) Hilbert space, let $R_{0}\in B\left(H\right)_{+}$,
and fix contractions 
\[
C_{1},\dots,C_{m}\in B\left(H\right).
\]
For each $j\in\left\{ 1,\dots,m\right\} $ we consider the weighted
residual (WR) update 
\[
\Phi_{j}\left(R\right):=R^{1/2}\left(I-C^{*}_{j}C_{j}\right)R^{1/2}\in B\left(H\right)_{+}.
\]
Iterating these updates along finite words produces a rooted $m$-ary
tree of residual operators 
\[
R_{\emptyset}:=R_{0},\qquad R_{wj}:=\Phi_{j}\left(R_{w}\right),
\]
together with dissipated pieces along the edges, 
\[
D_{w,j}:=R_{w}-R_{wj}=R^{1/2}_{w}C^{*}_{j}C_{j}R^{1/2}_{w}\in B\left(H\right)_{+}.
\]
The one-step identity $R_{w}=D_{w,j}+R_{wj}$ is elementary, but its
repeated use organizes the dynamics into a canonical telescoping decomposition
along every branch. The purpose of this paper is to attach to this
energy tree a natural family of probability measures on a path space,
and to identify the resulting boundary object that governs the asymptotic
behavior of the residuals. What makes the construction genuinely operator-driven
is that the transition rules are generated from the evolving residuals
themselves, rather than imposed externally on the tree.

The main results are as follows: (i) the construction of two intrinsic
path measures $\nu_{x}$ and $\nu_{\mathrm{tr}}$ on an absorbed path
space $\Omega$ determined by the WR energy tree; (ii) quantitative
extinction estimates for the associated scalar residual processes
under uniform leakage hypotheses; (iii) under the trace-class hypothesis
$R_{0}\in S_{1}\left(H\right)_{+}$, domination $\nu_{x}\ll\nu_{\mathrm{tr}}$
for every $x$ and an identification of $d\nu_{x}/d\nu_{\mathrm{tr}}$
as the canonical $\nu_{\mathrm{tr}}$-martingale limit of finite-level
likelihood ratios on cylinders; and (iv) a boundary disintegration
of $\nu_{\mathrm{tr}}$ over the terminal trace-class variable $R_{\infty}$,
together with a boundary representation of each $\nu_{x}$ via a density
$h_{x}$ on $S_{1}\left(H\right)_{+}$.

A basic feature of contraction splittings is that the process may
terminate. At a node $w$ it can happen that all edge energies vanish
for a given bias functional, so there is no longer a meaningful choice
among $\left\{ 1,\dots,m\right\} $. Rather than patching in an auxiliary
probability vector at such nodes, we adjoin an absorbing symbol $0$
and work on the enlarged alphabet $\mathcal{A}:=\left\{ 0,1,\dots,m\right\} $
with the convention that once $0$ occurs it persists. This choice
keeps the cylinder-set bookkeeping unchanged while treating termination
as part of the intrinsic dynamics. (One can think of $0$ as a bookkeeping
device that turns ``dead nodes'' into honest stopping times.)

Two bias functionals play distinct roles. Fix $x\in H$ with $\left\langle x,R_{0}x\right\rangle >0$.
The scalar edge energies 
\[
e_{w,j}\left(x\right):=\left\langle x,D_{w,j}x\right\rangle 
\]
determine a family of transition probabilities and hence a path measure
$\nu_{x}$ on $\Omega=\mathcal{A}^{\mathbb{N}}$. This vector-biased
law is intrinsic in $H$ and depends on the chosen starting vector.
In the trace-class setting $R_{0}\in S_{1}\left(H\right)_{+}$ there
is also a distinguished bias functional independent of a choice of
vector: 
\[
e^{\left(\mathrm{tr}\right)}_{w,j}:=\mathrm{tr}\left(D_{w,j}\right),
\]
leading to a reference path measure $\nu_{\mathrm{tr}}$. The measures
$\nu_{x}$ and $\nu_{\mathrm{tr}}$ are constructed from the same
energy tree, but they reflect different notions of dissipation along
edges, and a priori they need not be related.

The first part of the paper develops the deterministic bookkeeping
needed for these constructions. Along any infinite path $\omega\in\Omega$
we define the residual and dissipation processes $\left\{ R_{n}\left(\omega\right)\right\} $
and $\left\{ D_{n}\left(\omega\right)\right\} $, prove the finite
telescoping identity 
\[
R_{0}=R_{n}\left(\omega\right)+\sum^{n}_{k=1}D_{k}\left(\omega\right),
\]
and obtain monotonicity $R_{n}\left(\omega\right)\downarrow R_{\infty}\left(\omega\right)$
in the strong operator topology. When $R_{0}\in S_{1}\left(H\right)_{+}$,
the residuals and dissipations remain trace class and the same telescoping
identity holds in trace norm. This produces the terminal trace-class
random variable $R_{\infty}:\Omega\to S_{1}\left(H\right)_{+}$ under
$\nu_{\mathrm{tr}}$, which we treat as the boundary variable of the
WR energy tree. The associated tail $\sigma$-field is 
\[
\mathcal{B}_{\mathrm{WR}}:=\sigma\left(R_{\infty}\right),
\]
equivalently $\sigma\left(R_{0}-R_{\infty}\right)$.

The second part of the paper identifies the relation between the two
intrinsic path laws. Under the trace-class hypothesis, we show that
$\nu_{\mathrm{tr}}$ dominates the family $\left\{ \nu_{x}\right\} $
and that the Radon-Nikodym derivative $\rho_{x}=d\nu_{x}/d\nu_{\mathrm{tr}}$
is obtained as the $\nu_{\mathrm{tr}}$-almost sure limit of the finite-level
likelihood ratios on cylinders. In particular, $\left\{ \rho_{x,n}\right\} $
forms a canonical martingale with respect to the coordinate filtration,
and $\rho_{x}$ is its almost sure limit. This gives a concrete change-of-measure
principle on the WR path space that will be used repeatedly.

The final part establishes boundary disintegration. Let $\mu_{\mathrm{tr}}:=\left(R_{\infty}\right)_{\#}\nu_{\mathrm{tr}}$
be the law of the boundary variable. Since $\Omega$ is compact metric
and $S_{1}\left(H\right)_{+}$ is Polish in the trace norm topology,
standard disintegration yields a measurable family of conditional
path measures $\left\{ \nu^{T}_{\mathrm{tr}}\right\} _{T\in S_{1}\left(H\right)_{+}}$
supported on the fibers of $R_{\infty}$, with 
\[
\nu_{\mathrm{tr}}\left(E\right)=\int_{S_{1}\left(H\right)_{+}}\nu^{T}_{\mathrm{tr}}\left(E\right)d\mu_{\mathrm{tr}}\left(T\right)
\]
for Borel $E\subset\Omega$. Combining this with the change of measure
produces a boundary representation for each $\nu_{x}$: there is a
boundary density $h_{x}\in L^{1}\left(\mu_{\mathrm{tr}}\right)$,
characterized by 
\[
d\mu_{x}=h_{x}d\mu_{\mathrm{tr}},\qquad\mu_{x}:=\left(R_{\infty}\right)_{\#}\nu_{x},
\]
such that 
\[
\nu_{x}\left(E\right)=\int_{S_{1}\left(H\right)_{+}}h_{x}\left(T\right)\nu^{T}_{\mathrm{tr}}\left(E\right)d\mu_{\mathrm{tr}}\left(T\right).
\]
Thus the trace-biased law $\nu_{\mathrm{tr}}$ serves as a single
reference measure that organizes the entire vector-biased family through
boundary densities and conditional path laws. From this perspective,
the pair $\left(\mu_{\mathrm{tr}},\left\{ \nu^{T}_{\mathrm{tr}}\right\} \right)$
and the family $\left\{ h_{x}\right\} $ constitute the intrinsic
probabilistic data attached to the WR energy tree.

The present work is part of a sequence. In an earlier paper (projection-only
splittings) the tree was generated by projection data and the associated
path measures were built from projection energies. The contraction
setting treated here is strictly more flexible: the residual update
is nonlinear in $R$ and the edge labels depend on the contractions
through $C^{*}_{j}C_{j}$. The absorbing symbol formalism gives a
clean treatment of termination, and the trace-biased reference law
leads to a boundary disintegration that is well adapted to operator-theoretic
limits. We have chosen to keep the present paper focused on the commutative
path-space boundary; a noncommutative boundary theory for WR dynamics,
in the sense of operator-algebraic Poisson boundaries or Shilov-type
boundaries, will be addressed separately.

\subsection*{Relation to the literature}

At the operator-theoretic level, the update $\Phi_{j}(R)=R^{1/2}\left(I-C^{*}_{j}C_{j}\right)R^{1/2}$
is a nonlinear residual transform on $B(H)_{+}$ built from a contraction
$C_{j}$ (equivalently from the effect $A_{j}=C^{*}_{j}C_{j}$). This
places the present work in the general circle of ideas around order-theoretic
splittings of positive operators, including (in the projection and
``effect'' settings) variants of shorting/parallel-sum constructions
and their range-additivity phenomena; see, for example, \cite{MR3284802,MR2214409}
for modern operator-theoretic treatments connected to shorted operators
and related identities.

On the probabilistic side, our path measures are built from cylinder
weights on an $m$-ary tree, followed by Kolmogorov extension, and
later a disintegration over the terminal random variable $R_{\infty}$.
All of these steps are standard in measure-theoretic probability on
Polish spaces, but the novelty here is the specific \emph{operator-driven}
transition structure: at each node $w$ the next step is biased by
an intrinsic energy functional computed from the dissipated pieces
$D_{w,j}$. For background on regular conditional probabilities and
disintegration in Polish/Radon settings, see \cite{MR4226142,MR770643,MR3987300,MR226684,MR1873379,MR2267655}.

The boundary variable $R_{\infty}$ and the $\sigma$-field $\sigma(R_{\infty})$
that organizes the intrinsic measures $\left\{ \nu_{x}\right\} $
via the trace-biased reference law $\nu_{\mathrm{tr}}$ are also naturally
compared with boundary theories for Markov chains and random walks
on trees/graphs, where bounded harmonic functions admit Poisson-type
representations in terms of boundary data; see, e.g., \cite{MR3616205,MR704539,MR1178986,MR1246471,MR1743100}.
A key distinction is that here the chain is not specified by an ambient
graph alone: the transition probabilities are generated \emph{from
the evolving residuals} and hence are themselves part of the operator
dynamics.

Finally, the use of contractions and trace/Hilbert-Schmidt energies
aligns the WR tree with themes from completely positive maps and quantum
probability. In particular, trace-biasing is closely related to the
canonical ``energy'' that appears in Kraus/Stinespring representations
and in quantum trajectory models; see, for example, \cite{MR2329688,MR2369315,MR1805844,MR2065240}.
In a different but conceptually adjacent direction, noncommutative
Poisson boundaries attached to completely positive maps and to (discrete/locally
compact) quantum groups provide a mature boundary theory in which
harmonic elements are selected by a Markov operator on a noncommutative
algebra; see \cite{MR1916370,MR2400727,MR2660685,MR2335776,MR2200270,MR3048010}.
While the present paper stays within a commutative path-space boundary
disintegration, these operator-algebraic boundary theories provide
a natural point of comparison and suggest further developments.

\subsection*{Organization}

\noindent\prettyref{sec:2} develops the contraction energy tree
and the basic WR splittings. \prettyref{sec:3} constructs the path
space with absorption and defines the intrinsic measures $\nu_{x}$
and $\nu_{\mathrm{tr}}$ from cylinder weights. \prettyref{sec:4}
studies the residual and dissipation processes along a single path
and proves the telescoping identities and monotone limits. Sections
\ref{sec:5}-\ref{sec:6} give quantitative control of the tail in
the vector-biased and trace-biased settings and establish the change-of-measure
theorem $d\nu_{x}/d\nu_{\mathrm{tr}}$ as a martingale limit. \prettyref{sec:7}
introduces the boundary variable $R_{\infty}$, proves the disintegration
of $\nu_{\mathrm{tr}}$ over $\sigma\left(R_{\infty}\right)$, and
derives the boundary representation of $\nu_{x}$ via the boundary
density $h_{x}$.

\section{Positive splittings and monotone limits}\label{sec:2}

This section collects the deterministic telescoping calculus underlying
WR-type iterations. We begin with the projection case and the basic
monotone-tower mechanism: successive WR updates produce a Loewner-decreasing
residual sequence with a canonical strong-operator limit, and hence
a finite-level telescoping identity with a well-defined remainder
(\prettyref{lem:b-1} and \prettyref{rem:b-2}). With this in place,
we then shift to the general setting needed for the sequel, in which
the same splitting-and-telescoping is driven by contractions rather
than projections; the precise parametrization and the corresponding
telescoping statement are stated in \prettyref{thm:b-3} and \prettyref{prop:2-5}.
\begin{lem}
\label{lem:b-1}Let $R_{0}\in B\left(H\right)_{+}$ and let $\left\{ P_{j}\right\} ^{\infty}_{j=1}$
be a sequence of orthogonal projections in $B(H)$. Define recursively
\begin{equation}
R_{j}:=R^{1/2}_{j-1}\left(I-P_{j}\right)R^{1/2}_{j-1}\in B\left(H\right)_{+},\qquad j\ge1.\label{eq:b-1}
\end{equation}
Then there exists a unique $R_{\infty}\in B(H)_{+}$ such that $R_{j}\overset{s}{\to}R_{\infty}$
and 
\begin{equation}
R_{0}\overset{s}{=}R_{\infty}+\sum^{\infty}_{j=1}R^{1/2}_{j-1}P_{j}R^{1/2}_{j-1},\label{eq:b-2}
\end{equation}
i.e., the partial sums $\sum^{n}_{j=1}R^{1/2}_{j-1}P_{j}R^{1/2}_{j-1}$
converge strongly to $R_{0}-R_{\infty}$. In particular, for all $x\in H$,
\begin{equation}
\Vert R^{1/2}_{0}x\Vert^{2}=\Vert R^{1/2}_{\infty}x\Vert^{2}+\sum^{\infty}_{j=1}\Vert P_{j}R^{1/2}_{j-1}x\Vert^{2}.\label{eq:b-3}
\end{equation}
\end{lem}

\begin{proof}
Fix $j\ge1$ and $x\in H$. Since $P_{j}$ and $I-P_{j}$ have orthogonal
ranges, 
\begin{equation}
\Vert R^{1/2}_{j-1}x\Vert^{2}=\Vert\left(I-P_{j}\right)R^{1/2}_{j-1}x\Vert^{2}+\Vert P_{j}R^{1/2}_{j-1}x\Vert^{2}.\label{eq:b-4}
\end{equation}
By \eqref{eq:b-1}, 
\begin{align*}
\Vert\left(I-P_{j}\right)R^{1/2}_{j-1}x\Vert^{2} & =\langle x,R^{1/2}_{j-1}\left(I-P_{j}\right)R^{1/2}_{j-1}x\rangle\\
 & =\left\langle x,R_{j}x\right\rangle =\Vert R^{1/2}_{j}x\Vert^{2},
\end{align*}
using the fact that $R_{j}\geq0$ so it has a positive square-root
$R^{1/2}_{j}$. Hence \eqref{eq:b-4} is equivalent to
\begin{align*}
\left\langle x,R_{j-1}x\right\rangle  & =\left\langle x,R_{j}x\right\rangle +\langle x,R^{1/2}_{j-1}P_{j}R^{1/2}_{j-1}x\rangle.
\end{align*}
Summing from $j=1$ to $n$ yields, for every $x\in H$, 
\[
\left\langle x,R_{0}x\right\rangle =\left\langle x,R_{n}x\right\rangle +\sum^{n}_{j=1}\langle x,R^{1/2}_{j-1}P_{j}R^{1/2}_{j-1}x\rangle,
\]
and therefore 
\[
R_{0}=R_{n}+\sum^{n}_{j=1}R^{1/2}_{j-1}P_{j}R^{1/2}_{j-1}\quad\text{in }B\left(H\right)_{+}.
\]
Next, since $0\le I-P_{j}\le I$, we have $0\le R_{j}\le R_{j-1}\le R_{0}$
for all $j$, so $(R_{n})$ is a bounded decreasing sequence in $B\left(H\right)_{+}$.
Hence $\left\langle x,R_{n}x\right\rangle $ decreases to a limit
for each $x$, and there is a unique $R_{\infty}\in B(H)_{+}$ such
that $R_{n}\xrightarrow{s}R_{\infty}$ (see e.g., \cite{MR493419}).
Using the finite identity, 
\[
\sum^{n}_{j=1}R^{1/2}_{j-1}P_{j}R^{1/2}_{j-1}=R_{0}-R_{n}\xrightarrow{s}R_{0}-R_{\infty},
\]
which gives \eqref{eq:b-2}. The identity \eqref{eq:b-3} follows
from this. 
\end{proof}
\begin{rem}
\label{rem:b-2}The preceding WR iteration should be contrasted with
the ``bare'' Kaczmarz/Pythagoras expansion (see, e.g., \cite{MR2311862,MR34514,MR2903120,MR3773065,MR2580440})
\begin{align*}
\left\Vert x\right\Vert ^{2} & =\left\Vert \left(I-P_{1}\right)x\right\Vert ^{2}+\left\Vert P_{1}x\right\Vert ^{2}\\
 & =\left\Vert \left(I-P_{2}\right)\left(I-P_{1}\right)x\right\Vert ^{2}+\left\Vert P_{2}\left(I-P_{1}\right)x\right\Vert ^{2}+\left\Vert P_{1}x\right\Vert ^{2}\\
 & \cdots
\end{align*}
Formally, this gives for each $n$ the exact finite decomposition
\[
\left\Vert x\right\Vert ^{2}=\left\Vert \left(I-P_{n}\right)\cdots\left(I-P_{1}\right)x\right\Vert ^{2}+\sum^{n}_{j=1}\left\Vert P_{j}\left(I-P_{j-1}\right)\cdots\left(I-P_{1}\right)x\right\Vert ^{2}.
\]
However, there is in general no reason for the ``tail term'' 
\begin{equation}
\left(I-P_{n}\right)\cdots\left(I-P_{1}\right)x\label{eq:b-5}
\end{equation}
to converge as $n\to\infty$. Products of orthogonal projections can
oscillate, and convergence typically requires additional geometric
input (angle conditions, regularity hypotheses, or special structure
such as commuting/averaged projections); without such hypotheses,
the finite telescoping identity does not come with a canonical limiting
remainder. 

A substantial literature studies precisely when cyclic (or more generally
non-periodic) products of orthogonal projections converge, and how
this depends on the mutual geometry of the underlying subspaces. One
convenient set of sufficient conditions is formulated in terms of
angle/regularity criteria guaranteeing uniform convergence for averaged
projections and for cyclic or random products; see \cite{MR3773065}.
Related convergence results for non-periodic infinite products of
orthogonal projections (and, more generally, nonexpansive operators)
are developed in \cite{MR2903120,MR3145756}, while quantitative error
bounds for simultaneous projection schemes with infinitely many subspaces
appear in \cite{MR4310540}. For a concise discussion of the general
phenomenon that iterated projection products may fail to converge
without additional hypotheses, see also \cite{MR2252935}.

The WR iteration avoids this instability by reweighting after each
split. Instead of repeatedly applying $\left(I-P_{j}\right)$ to the
fixed vector $x$, one applies $\left(I-P_{j}\right)$ to the evolving
vector $R^{1/2}_{j-1}x$, where $R_{j}$ is updated by \eqref{eq:b-1}.
This produces a monotone tower $0\le R_{j}\le R_{j-1}\le R_{0}$,
hence $R_{j}\xrightarrow{s}R_{\infty}$, and the ``tail energy''
$\Vert R^{1/2}_{j}x\Vert^{2}$ always converges. Equivalently, the
residual operators $R^{1/2}_{j-1}P_{j}R^{1/2}_{j-1}$ form a strongly
convergent series, yielding the canonical decomposition \eqref{eq:b-2}.
In short, Kaczmarz controls a sequence of vectors \eqref{eq:b-5},
whose limit may fail to exist, whereas WR controls a Loewner-decreasing
sequence $R_{n}\downarrow R_{\infty}$, so the remainder term is forced
to converge.
\end{rem}

It is useful to observe that the WR iteration step is only one instance
of a more general positive splitting mechanism. The point is not the
projection itself, but the fact that the update produces a decomposition
of a positive operator into two positive pieces, one of which becomes
the new remainder. This suggests enlarging the class of admissible
“splitting operators” beyond projections, while keeping the same telescoping
and strong-limit features of \prettyref{lem:b-1}. \prettyref{thm:b-3}
identifies the correct level of generality and provides the converse
statement that any positive splitting arises in this way; \prettyref{cor:b-4}
then recovers the projection case.
\begin{thm}
\label{thm:b-3}Let $H$ be a Hilbert space and $R\in B(H)_{+}$. 
\begin{enumerate}
\item If $C\in B(H)$ is a contraction, define 
\[
D:=R^{1/2}C^{*}CR^{1/2},\qquad R_{1}:=R^{1/2}\left(I-C^{*}C\right)R^{1/2}.
\]
Then $D,R_{1}\in B\left(H\right)_{+}$, and for every $x\in H$, 
\begin{align*}
\langle x,Rx\rangle & =\langle x,R_{1}x\rangle+\langle x,Dx\rangle\\
 & =\Vert\left(I-C^{*}C\right)^{1/2}R^{1/2}x\Vert^{2}+\Vert CR^{1/2}x\Vert^{2}.
\end{align*}
\item Conversely, if $A,B\in B(H)_{+}$ with $A+B=R$, then there exists
a contraction $C\in B(H)$, unique up to the choice of the partial
isometry in the Douglas factorization on $\overline{ran}\left(A^{1/2}\right)\subset\overline{ran}(R^{1/2})$,
such that 
\[
A=R^{1/2}C^{*}CR^{1/2},\qquad B=R^{1/2}\left(I-C^{*}C\right)R^{1/2}.
\]
\end{enumerate}
\end{thm}

\begin{proof}
Fix a contraction $C\in B\left(H\right)$, and set 
\[
Q:=I-C^{*}C\ge0.
\]
By functional calculus, $Q^{1/2}\in B(H)_{+}$. For any $y\in H$,
\[
\left\Vert y\right\Vert ^{2}=\langle y,y\rangle=\langle y,\left(C^{*}C+Q\right)y\rangle=\Vert Q^{1/2}y\Vert^{2}+\left\Vert Cy\right\Vert ^{2}.
\]
Now fix $x\in H$ and take $y=R^{1/2}x$. Then 
\[
\Vert R^{1/2}x\Vert^{2}=\Vert Q^{1/2}R^{1/2}x\Vert^{2}+\Vert CR^{1/2}x\Vert^{2}.
\]
The first term can be written as 
\begin{align*}
\Vert Q^{1/2}R^{1/2}x\Vert^{2} & =\langle x,R^{1/2}QR^{1/2}x\rangle=\left\langle x,R_{1}x\right\rangle ,
\end{align*}
and similarly 
\[
\Vert CR^{1/2}x\Vert^{2}=\langle x,R^{1/2}C^{*}CR^{1/2}x\rangle=\left\langle x,Dx\right\rangle .
\]
On the other hand $\Vert R^{1/2}x\Vert^{2}=\left\langle x,Rx\right\rangle $.
Thus, for all $x\in H$, 
\begin{align*}
\left\langle x,Rx\right\rangle  & =\left\langle x,R_{1}x\right\rangle +\left\langle x,Dx\right\rangle \\
 & =\Vert\left(I-C^{*}C\right)^{1/2}R^{1/2}x\Vert^{2}+\Vert CR^{1/2}x\Vert^{2}.
\end{align*}
This shows $R=R_{1}+D$ in the sense of quadratic forms. Both $R_{1}$
and $D$ are of the form $T^{*}T$ for some $T\in B(H)$ (namely $T=Q^{1/2}R^{1/2}$
and $T=CR^{1/2}$), hence $R_{1},D\ge0$.

Conversely, suppose $A,B\in B\left(H\right)_{+}$ and $A+B=R$. Because
$B=R-A\ge0$, we have $A\le R$ in Loewner order. Applying Douglas'
factorization lemma \cite{MR203464} to the inequality $A\le R$,
there exists a contraction $C\in B(H)$ such that 
\[
A^{1/2}=CR^{1/2}.
\]
Then 
\[
A=A^{1/2}A^{1/2}=R^{1/2}C^{*}CR^{1/2},
\]
and consequently 
\[
B=R-A=R^{1/2}\left(I-C^{*}C\right)R^{1/2}.
\]

For uniqueness: if $C_{1},C_{2}$ are contractions with $R^{1/2}C^{*}_{i}C_{i}R^{1/2}=A$,
then 
\[
(C_{i}R^{1/2})^{*}(C_{i}R^{1/2})=A,\qquad i=1,2.
\]
Thus $C_{1}R^{1/2}$ and $C_{2}R^{1/2}$ have the same absolute value
$A^{1/2}$, so by polar decomposition they differ by a partial isometry
on $\overline{ran}(A^{1/2})\subseteq\overline{ran}(R^{1/2})$. Equivalently,
the ambiguity is exactly the choice of the partial isometry factor
in the Douglas representation $A^{1/2}=CR^{1/2}$. 
\end{proof}
\begin{cor}
\label{cor:b-4}The projection case $C=P$ gives the WR step 
\[
R_{1}=R^{1/2}(I-P)R^{1/2},\qquad D=R^{1/2}PR^{1/2}.
\]
\end{cor}

Since \prettyref{thm:b-3} identifies contractions as the correct
objects parametrizing positive splittings, the next step is to iterate
these splittings exactly as in \prettyref{lem:b-1}. The resulting
recursion produces the same telescoping identities and strong-limit
remainder, but now with projections replaced by an arbitrary sequence
of contractions. We include this version for later use.
\begin{prop}
\label{prop:2-5}Let $R_{0}\in B(H)_{+}$ and let $\left\{ C_{j}\right\} ^{\infty}_{j=1}\subset B(H)$
be a sequence of contractions. Define recursively 
\begin{equation}
R_{j}:=R^{1/2}_{j-1}\left(I-C^{*}_{j}C_{j}\right)R^{1/2}_{j-1}\in B(H)_{+},\qquad j\ge1.
\end{equation}
Then for every $n\ge1$, 
\begin{equation}
R_{0}=R_{n}+\sum^{n}_{j=1}R^{1/2}_{j-1}C^{*}_{j}C_{j}R^{1/2}_{j-1}\quad\text{in }B\left(H\right).
\end{equation}
Moreover, there exists a unique $R_{\infty}\in B(H)_{+}$ such that
$R_{j}\xrightarrow{s}R_{\infty}$ and 
\begin{equation}
R_{0}\overset{s}{=}R_{\infty}+\sum^{\infty}_{j=1}R^{1/2}_{j-1}C^{*}_{j}C_{j}R^{1/2}_{j-1}.
\end{equation}
In particular, for all $x\in H$, 
\begin{equation}
\Vert R^{1/2}_{0}x\Vert^{2}=\Vert R^{1/2}_{\infty}x\Vert^{2}+\sum^{\infty}_{j=1}\Vert C_{j}R^{1/2}_{j-1}x\Vert^{2}.
\end{equation}
\end{prop}

\begin{proof}
The argument is identical to the proof of \prettyref{lem:b-1}: replace
$P_{j}$ by $C^{*}_{j}C_{j}$ and note that $0\leq I-C^{*}_{j}C_{j}\leq I$.
It follows that $R_{j}$ is again bounded and decreasing in $B\left(H\right)_{+}$,
and therefore has an SOT limit. 
\end{proof}

\section{Energy trees and path measures}\label{sec:3}

In this section we set up the basic probabilistic notation and constructions
associated with iterated contraction splittings. The goal is to define
the energy tree, to specify cylinder weights (hence probabilities)
on the path space, and to apply Kolmogorov's extension theorem to
obtain the resulting path measures. One technical point is treated
differently from \cite{tian20261}: whenever the splitting mechanism
terminates at a node, we do not patch in an auxiliary probability
vector. Instead we adjoin a terminating symbol $0$ and declare it
absorbing, forcing the process to remain in the absorbing state once
termination occurs. This convention treats termination as an absorption
(stopping) event on the path space, and it will be used later when
we relate boundary disintegrations to stopped paths and survival probabilities.

\subsection{The energy tree}

Fix $m\ge1$, a Hilbert space $H$, an initial operator $R_{0}\in B\left(H\right)_{+}$,
and contractions $C_{1},\dots,C_{m}\in B\left(H\right)$. For each
$j\in\left\{ 1,\dots,m\right\} $ define the residual update 
\[
\Phi_{j}(R):=R^{1/2}\left(I-C^{*}_{j}C_{j}\right)R^{1/2}\in B\left(H\right)_{+}.
\]
Let $\mathcal{W}$ denote the set of finite words over $\left\{ 1,\dots,m\right\} $,
including the empty word $\emptyset$. For a word $w=w_{1}\cdots w_{n}\in\mathcal{W}$
we define the residual operator inductively by 
\begin{equation}
R_{\emptyset}:=R_{0},\qquad R_{wj}:=\Phi_{j}\left(R_{w}\right)\quad\text{for }j\in\left\{ 1,\dots,m\right\} .\label{eq:3-1}
\end{equation}
Thus each node $w$ has $m$ children $w1,\dots,wm$, and $R_{w}$
is the residual at that node. Associated to the edge $w\to wj$ we
have the dissipated piece 
\begin{equation}
D_{w,j}:=R_{w}-R_{wj}=R^{1/2}_{w}C^{*}_{j}C_{j}R^{1/2}_{w}\in B\left(H\right)_{+}.\label{eq:3-2}
\end{equation}
In particular, for each $w$ and $j$, the one-step splitting identity
holds: 
\begin{equation}
R_{w}=D_{w,j}+R_{wj}.\label{eq:3.3}
\end{equation}
We refer to the rooted $m$-ary tree equipped with the labels $w\mapsto R_{w}$
and $\left(w,j\right)\mapsto D_{w,j}$ as the \textit{contraction
energy tree} generated by $(R_{0},\left\{ C_{j}\right\} ^{m}_{j=1})$.

\subsection{Path space and absorption}

Let $\mathcal{A}:=\left\{ 0,1,\dots,m\right\} $, and set 
\begin{equation}
\Omega:=\mathcal{A}^{\mathbb{N}}.\label{eq:3-4}
\end{equation}
Elements $\omega=\left(\omega_{1},\omega_{2},\dots\right)\in\Omega$
are infinite sequences of symbols. The interpretation is that symbols
in $\left\{ 1,\dots,m\right\} $ represent genuine contraction choices,
while the symbol $0$ represents termination. We impose the absorbing
convention: once $0$ occurs, it persists. Formally, we will define
cylinder probabilities so that whenever a finite prefix contains $0$,
all mass is assigned to extending that prefix by $0$ again.

For a finite word $u=u_{1}\cdots u_{n}$ over $\mathcal{A}$, we write
$\left[u\right]\subset\Omega$ for the corresponding cylinder set
\[
\left[u\right]:=\left\{ \omega\in\Omega:\omega_{1}=u_{1},\dots,\omega_{n}=u_{n}\right\} .
\]

It is worth noting that the path space \eqref{eq:3-4} is different
from
\[
\left\{ 1,\dots,m\right\} ^{\mathbb{N}}\cup\left\{ \dagger\right\} 
\]
with a single adjoined cemetery point $\dagger$. The latter alternative
only distinguishes “survival” from “termination” whereas the absorbing
symbol model retains when termination occurs via the first index $\tau(\omega):=\inf\left\{ n\ge1:\omega_{n}=0\right\} \in\left\{ 1,2,\dots\right\} \cup\left\{ \infty\right\} $.

Equivalently, a finite word $w=w_{1}\cdots w_{n}\in\left\{ 1,\dots,m\right\} ^{n}$
is identified with the infinite sequence $w_{1}\cdots w_{n}000\cdots\in\mathcal{A}^{\mathbb{N}}$.
This convention keeps the cylinder set bookkeeping and Kolmogorov
consistency unchanged while treating termination intrinsically through
absorption.

\subsection{Energy-biased cylinders}

Fix $x\in H$ with $\left\langle x,R_{0}x\right\rangle >0$. For each
node $w\in\mathcal{W}$ define the scalar edge-energies 
\[
e_{w,j}\left(x\right):=\left\langle x,D_{w,j}x\right\rangle =\langle x,R^{1/2}_{w}C^{*}_{j}C_{j}R^{1/2}_{w}x\rangle,\qquad j\in\left\{ 1,\dots,m\right\} .
\]
Set 
\[
s_{w}\left(x\right):=\sum^{m}_{j=1}e_{w,j}\left(x\right)\in[0,\infty).
\]
We say that $w$ is \textit{dead for $x$} if $s_{w}\left(x\right)=0$
(thus all $e_{w,j}\left(x\right)$ vanish).

We now define a family of transition probabilities at each node. If
$s_{w}\left(x\right)>0$, set 
\begin{equation}
p^{\left(x\right)}_{w,j}:=\begin{cases}
0 & j=0,\\
\frac{e_{w,j}\left(x\right)}{s_{w}\left(x\right)} & j\in\left\{ 1,\dots,m\right\} .
\end{cases}\label{eq:3-5}
\end{equation}
If $s_{w}\left(x\right)=0$, set 
\begin{equation}
p^{\left(x\right)}_{w,j}:=\begin{cases}
1 & j=0,\\
0 & j\in\left\{ 1,\dots,m\right\} .
\end{cases}\label{eq:3-6}
\end{equation}
Finally, to enforce absorption, if a finite word $u$ over $\mathcal{A}$
already contains $0$, we declare 
\begin{equation}
p^{\left(x\right)}_{u,0}:=1,\qquad p^{\left(x\right)}_{u,j}:=0\quad\text{for }j\in\left\{ 1,\dots,m\right\} ,\label{eq:3-7}
\end{equation}
so that the process remains at $0$ forever after it terminates.

Define cylinder weights $\mu_{x}\left(\left[u\right]\right)$ recursively
by 
\[
\mu_{x}\left(\left[\emptyset\right]\right):=1,\qquad\mu_{x}\left(\left[ua\right]\right):=\mu_{x}\left(\left[u\right]\right)p^{\left(x\right)}_{u,a}\quad\text{for }a\in\mathcal{A}.
\]
By construction, for every finite word $u$ we have $\sum_{a\in\mathcal{A}}p^{\left(x\right)}_{u,a}=1$,
hence 
\[
\sum_{a\in\mathcal{A}}\mu_{x}\left(\left[ua\right]\right)=\mu_{x}\left(\left[u\right]\right).
\]
Therefore $\left\{ \mu_{x}\left(\left[u\right]\right)\right\} $ is
a consistent family of cylinder probabilities.
\begin{defn}
\label{def:3-1}The \textit{energy-biased path measure} $\nu_{x}$
is the unique probability measure on $\Omega$ whose values on cylinder
sets satisfy $\nu_{x}\left(\left[u\right]\right)=\mu_{x}\left(\left[u\right]\right)$
for all finite words $u$ over $\mathcal{A}$. (Existence and uniqueness
follow from Kolmogorov's extension theorem.) 
\end{defn}

\subsection{Trace-biased cylinders}

Assume now that $R_{0}$ is trace class. Then $R^{1/2}_{0}$ is Hilbert-Schmidt,
hence for every bounded operator $B\in B\left(H\right)$ the operator
$R^{1/2}_{0}BR^{1/2}_{0}$ is trace class. In particular, for each
$j\in\left\{ 1,\dots,m\right\} $ the residual update $\Phi_{j}\left(R_{0}\right)=R^{1/2}_{0}\left(I-C^{*}_{j}C_{j}\right)R^{1/2}_{0}$
is trace class. By induction on the length $\left|w\right|$, it follows
that every $R_{w}$ and every $D_{w,j}=R^{1/2}_{w}C^{*}_{j}C_{j}R^{1/2}_{w}$
is trace class. 

Define the trace edge-energies
\[
e^{\left(\mathrm{tr}\right)}_{w,j}:=\mathrm{tr}\left(D_{w,j}\right),\qquad s^{\left(\mathrm{tr}\right)}_{w}:=\sum^{m}_{j=1}e^{\left(\mathrm{tr}\right)}_{w,j}.
\]
The transition probabilities $p^{(\mathrm{tr})}_{w,a}$ are defined
exactly as in \eqref{eq:3-5}-\eqref{eq:3-7}, with $e_{w,j}\left(x\right)$
replaced by $e^{\left(\mathrm{tr}\right)}_{w,j}$. The same cylinder
recursion yields a consistent family of cylinder probabilities, hence
Kolmogorov's theorem produces a unique probability measure $\nu_{\mathrm{tr}}$
on $\Omega$.

This trace-biased choice is basis-free: it weights the outgoing edges
according to the amount of trace dissipated along that edge. In the
special case $\sum^{m}_{j=1}C^{*}_{j}C_{j}=I$, one has $s^{\left(\mathrm{tr}\right)}_{w}=\mathrm{tr}\left(R_{w}\right)$,
so the normalization is simply by the trace mass present at node $w$.
\begin{rem}
The energy-biased and trace-biased constructions are formally parallel
but conceptually different. The energy-biased weights $e_{w,j}\left(x\right)=\left\langle x,D_{w,j}x\right\rangle $
depend on a chosen vector $x$ and therefore encode how dissipation
is seen from the state $x$. By contrast, the trace-biased weights
$e^{\left(\mathrm{tr}\right)}_{w,j}=\mathrm{tr}\left(D_{w,j}\right)$
depend only on the operators and are unitarily invariant. We will
use both points of view: the $x$-biased measures naturally interface
with vector-level identities and disintegrations, while the trace-biased
measure provides an intrinsic averaging that is well-adapted to trace-class
hypotheses.
\end{rem}

\section{Pathwise telescoping and stopping times}\label{sec:4}

We now recast the deterministic telescoping identity along a single
infinite path. The only additional feature relative to the tree in
\prettyref{sec:3} is the absorbing symbol $0$, which gives a natural
stopping time.

Recall $\mathcal{A}=\left\{ 0,1,\dots,m\right\} $ and $\Omega=\mathcal{A}^{\mathbb{N}}$.
Let $\mathcal{A}^{\ast}$ be the set of finite words over $\mathcal{A}$,
including the empty word $\emptyset$. If $u=u_{1}\cdots u_{n}\in\mathcal{A}^{\ast}$,
we write $\left|u\right|=n$.

We first extend the residual labels $w\mapsto R_{w}$ (defined for
words $w\in\left\{ 1,\dots,m\right\} ^{\ast}$) to all finite words
over $\mathcal{A}$ by freezing at the first occurrence of $0$.
\begin{defn}
\label{def:d-1}If $u\in\mathcal{A}^{\ast}$ contains no $0$, then
$u\in\left\{ 1,\dots,m\right\} ^{\ast}$ and we keep the notation
$R_{u}$ from \eqref{eq:3-1}-\eqref{eq:3.3}. If $u$ contains a
$0$, write $u=v0u'$ where $v\in\left\{ 1,\dots,m\right\} ^{\ast}$
is the maximal prefix containing no $0$ (equivalently, the prefix
before the first $0$), and set 
\[
R_{u}:=R_{v}.
\]
\end{defn}

We also extend the dissipated pieces to the enlarged alphabet.
\begin{defn}
\label{def:d-2}For $u\in\mathcal{A}^{\ast}$ and $a\in\mathcal{A}$,
set 
\[
D_{u,a}:=\begin{cases}
R^{1/2}_{u}C^{*}_{a}C_{a}R^{1/2}_{u}, & a\in\left\{ 1,\dots,m\right\} \ \text{and }u\text{ contains no }0,\\[4pt]
0, & \text{otherwise}.
\end{cases}
\]
Then $D_{u,a}\in B\left(H\right)_{+}$ and, in particular, $D_{u,a}=0$
whenever $a=0$ or $u$ already contains a $0$. Thus after absorption
no further dissipation occurs, and the residual label remains frozen.
\end{defn}

With these conventions the one-step splitting identity holds uniformly.
\begin{lem}
\label{lem:d-3}For every $u\in\mathcal{A}^{\ast}$ and every $a\in\mathcal{A}$,
\[
R_{u}=D_{u,a}+R_{ua}.
\]
\end{lem}

\begin{proof}
If $u$ contains no $0$ and $a\in\left\{ 1,\dots,m\right\} $, this
is exactly the defining WR splitting from \eqref{eq:3-1}-\eqref{eq:3.3}.
If $a=0$, then $D_{u,0}=0$ and $R_{u0}=R_{u}$. If $u$ already
contains a $0$, then $R_{ua}=R_{u}$ by \prettyref{def:d-1} and
$D_{u,a}=0$ by \prettyref{def:d-2}.
\end{proof}
Now fix $\omega=\left(\omega_{1},\omega_{2},\dots\right)\in\Omega$,
and write $\omega|n:=\omega_{1}\cdots\omega_{n}\in\mathcal{A}^{\ast}$
for the length-$n$ prefix, with $\omega|0=\emptyset$.
\begin{defn}
The residual and dissipation processes along $\omega$ are $R_{n}\left(\omega\right):=R_{\omega|n}$,
$n\ge0$, and, for $n\ge1$, $D_{n}\left(\omega\right):=D_{\omega|n-1,\omega_{n}}$.
We also set 
\begin{equation}
\tau\left(\omega\right):=\inf\left\{ n\ge1:\omega_{n}=0\right\} \in\left\{ 1,2,\dots\right\} \cup\left\{ \infty\right\} \label{eq:4-1}
\end{equation}
for the (possibly infinite) termination time. 
\end{defn}

\begin{lem}
\label{lem:d-5}Let $\omega\in\Omega$ and set $\tau=\tau\left(\omega\right)$.
If $\tau<\infty$, then 
\[
R_{n}\left(\omega\right)=R_{\tau-1}\left(\omega\right)\quad\text{and}\quad D_{n}\left(\omega\right)=0\quad\text{for all }n\ge\tau.
\]
\end{lem}

\begin{proof}
If $\tau<\infty$, then for every $n\ge\tau$ we can write $\omega|n=v0u$
with $v=\omega|\left(\tau-1\right)$, so $v$ is the prefix before
the first $0$. \prettyref{def:d-1} gives $R_{\omega|n}=R_{v}$.
For the dissipation, $D_{n}\left(\omega\right)=0$ by \prettyref{def:d-2}
because either $\omega_{n}=0$ (if $n=\tau$) or $\omega|n-1$ already
contains $0$ (if $n>\tau$). 
\end{proof}
The telescoping identity along $\omega$ is now straightforward.
\begin{prop}
\label{prop:d-6}For every $\omega\in\Omega$ and every $n\ge1$,
\[
R_{0}=R_{n}\left(\omega\right)+\sum^{n}_{k=1}D_{k}\left(\omega\right).
\]
\end{prop}

\begin{proof}
Apply \prettyref{lem:d-3} with $u=\omega|k-1$ and $a=\omega_{k}$
for $k=1,\dots,n$, and sum the resulting identities.
\end{proof}
In particular, the process is monotone in Loewner order.
\begin{cor}
\label{cor:4-7}For every $\omega\in\Omega$ and every $n\ge1$, 
\[
0\le R_{n}\left(\omega\right)\le R_{n-1}\left(\omega\right)\le\cdots\le R_{0}.
\]
\end{cor}

\begin{proof}
\prettyref{prop:d-6} with $n$ and $n-1$ gives $R_{n-1}\left(\omega\right)=R_{n}\left(\omega\right)+D_{n}\left(\omega\right)$
with $D_{n}\left(\omega\right)\ge0$. 
\end{proof}
\begin{prop}
\label{prop:d-8}For every $\omega\in\Omega$, the limit 
\[
R_{\infty}\left(\omega\right):={\rm s\text{-}lim}_{n\to\infty}R_{n}\left(\omega\right)
\]
exists in $B\left(H\right)_{+}$.
\end{prop}

\begin{proof}
By \prettyref{cor:4-7}, the sequence $\left\{ R_{n}(\omega)\right\} _{n\ge0}$
is decreasing in Loewner order and satisfies $0\le R_{n}\left(\omega\right)\le R_{0}$
for all $n$. Hence $\left\{ R_{n}(\omega)\right\} $ converges in
the strong operator topology to a positive operator $R_{\infty}\left(\omega\right)\in B\left(H\right)_{+}$
by the monotone convergence theorem for bounded selfadjoint operators;
see \cite{MR493419}. 
\end{proof}
Combining finite telescoping with the strong limit gives the infinite
telescoping identity, with a finite truncation when the path terminates.
\begin{prop}
\label{prop:d-9}Let $\omega\in\Omega$ and $\tau=\tau\left(\omega\right)$.
If $\tau<\infty$, then $R_{\infty}\left(\omega\right)=R_{\tau-1}\left(\omega\right)$
and 
\[
R_{0}=R_{\infty}\left(\omega\right)+\sum^{\tau-1}_{k=1}D_{k}\left(\omega\right).
\]
If $\tau=\infty$, then 
\[
R_{0}\overset{s}{=}R_{\infty}\left(\omega\right)+\sum^{\infty}_{k=1}D_{k}\left(\omega\right),
\]
i.e., the partial sums $\sum^{n}_{k=1}D_{k}\left(\omega\right)$ converge
in SOT to $R_{0}-R_{\infty}\left(\omega\right)$.
\end{prop}

\begin{proof}
If $\tau<\infty$, \prettyref{lem:d-5} shows that $R_{n}\left(\omega\right)=R_{\tau-1}\left(\omega\right)$
and $D_{k}\left(\omega\right)=0$ for all $k\ge\tau$. Taking $n\ge\tau$
in \prettyref{prop:d-6} and dropping the zero terms yields the finite
identity. The strong limit must then equal $R_{\tau-1}\left(\omega\right)$.
If $\tau=\infty$, \prettyref{prop:d-6} gives 
\[
\sum^{n}_{k=1}D_{k}\left(\omega\right)=R_{0}-R_{n}\left(\omega\right),
\]
and taking $n\to\infty$ and using \prettyref{prop:d-8} yields SOT
convergence to $R_{0}-R_{\infty}\left(\omega\right)$. 
\end{proof}
\begin{cor}
\label{cor:d-10}For every $\omega\in\Omega$, $\tau=\tau\left(\omega\right)$,
and $x\in H$, 
\[
\left\langle x,R_{0}x\right\rangle =\left\langle x,R_{\infty}\left(\omega\right)x\right\rangle +\sum^{\tau-1}_{k=1}\left\langle x,D_{k}\left(\omega\right)x\right\rangle ,
\]
with the convention that if $\tau=\infty$ the sum runs over all $k\ge1$.
\end{cor}

\begin{proof}
Use \prettyref{prop:d-9} and the monotone convergence of the nonnegative
series $\sum_{k}\left\langle x,D_{k}\left(\omega\right)x\right\rangle $. 
\end{proof}

\section{Extinction estimates}\label{sec:5}

Fix $\Omega=\left\{ 0,1,\dots,m\right\} ^{\mathbb{N}}$ and let $\tau\left(\omega\right)$
be the termination time from \eqref{eq:4-1}. Along $\omega\in\Omega$
we have the pathwise residual and dissipation processes $R_{n}\left(\omega\right)$
and $D_{n}\left(\omega\right)$, and the stopped telescoping identity
\[
R_{0}=R_{n}\left(\omega\right)+\sum^{n}_{k=1}D_{k}\left(\omega\right),\qquad n\ge1.
\]
In particular, $R_{n}\left(\omega\right)\downarrow R_{\infty}\left(\omega\right)$
in SOT, and if $\tau\left(\omega\right)<\infty$ then the sum truncates
at $\tau\left(\omega\right)-1$ (\prettyref{prop:d-9}).

\subsection{Energy-biased extinction}\label{subsec:5-1}

Fix $x\in H$ with $\left\langle x,R_{0}x\right\rangle >0$, and let
$\nu_{x}$ be the energy-biased path measure (\prettyref{def:3-1}).
For a node $w\in\left\{ 1,\dots,m\right\} ^{\ast}$, recall the edge
energies 
\[
e_{w,j}\left(x\right):=\left\langle x,D_{w,j}x\right\rangle ,\qquad s_{w}\left(x\right):=\sum^{m}_{j=1}e_{w,j}\left(x\right).
\]
The energy-biased transitions at node $w$ are $p^{\left(x\right)}_{w,j}=e_{w,j}\left(x\right)/s_{w}\left(x\right)$
when $s_{w}\left(x\right)>0$, and if $s_{w}\left(x\right)=0$ then
the chain terminates (probability $1$ on the symbol $0$); see \eqref{eq:3-5}-\eqref{eq:3-7}.

Define the scalar residual process 
\[
M_{n}\left(\omega\right):=\left\langle x,R_{n}\left(\omega\right)x\right\rangle ,\qquad n\ge0.
\]
Let $\mathcal{F}_{n}:=\sigma\left(\omega_{1},\dots,\omega_{n}\right)$
be the natural filtration. We impose the following quantitative leakage
hypothesis.
\begin{assumption}[uniform $x$-leakage]
\label{assu:f-1} There exists $\alpha\in(0,m]$ such that for every
$w\in\left\{ 1,\dots,m\right\} ^{\ast}$, 
\[
s_{w}\left(x\right)\ge\alpha\left\langle x,R_{w}x\right\rangle .
\]
\end{assumption}

Note that $\alpha\le m$ is forced by the definitions. Indeed, for
each node $w$ and each $j\in\left\{ 1,\dots,m\right\} $,
\[
e_{w,j}\left(x\right)=\left\langle x,D_{w,j}x\right\rangle =\langle R^{1/2}_{w}x,\left(C^{*}_{j}C_{j}\right)R^{1/2}_{w}x\rangle\le\langle R^{1/2}_{w}x,R^{1/2}_{w}x\rangle=\left\langle x,R_{w}x\right\rangle 
\]
Summing over $j$ gives $s_{w}(x)=\sum^{m}_{j=1}e_{w,j}(x)\le m\left\langle x,R_{w}x\right\rangle $.
In particular, the constant $c=1-\alpha/m$ appearing later (\prettyref{thm:f-4})
satisfies $c\in\left[0,1\right)$.

The next lemma is the core computation.
\begin{lem}[conditional dissipation identity]
\label{lem:f-2} Let $n\ge1$ and condition on the event $\left\{ \omega|n-1=w\right\} $,
where $w\in\left\{ 1,\dots,m\right\} ^{\ast}$. Then 
\[
\mathbb{E}_{\nu_{x}}\left[\left\langle x,D_{n}\left(\omega\right)x\right\rangle \mid\mathcal{F}_{n-1}\right]=\begin{cases}
{\displaystyle \frac{\sum^{m}_{j=1}e_{w,j}\left(x\right)^{2}}{s_{w}\left(x\right)},} & s_{w}\left(x\right)>0,\\[10pt]
0, & s_{w}\left(x\right)=0.
\end{cases}
\]
\end{lem}

\begin{proof}
If $s_{w}\left(x\right)=0$, the transition is forced to the absorbing
symbol $0$, and by \prettyref{sec:4}, the dissipation at that step
is $0$. If $s_{w}\left(x\right)>0$, then $\omega_{n}\in\left\{ 1,\dots,m\right\} $
with probabilities $p^{\left(x\right)}_{w,j}=e_{w,j}\left(x\right)/s_{w}\left(x\right)$.
On $\left\{ \omega|n-1=w,\ \omega_{n}=j\right\} $ we have $D_{n}\left(\omega\right)=D_{w,j}$,
hence 
\[
\mathbb{E}_{\nu_{x}}\left[\left\langle x,D_{n}x\right\rangle \mid\omega|n-1=w\right]=\sum^{m}_{j=1}p^{\left(x\right)}_{w,j}\left\langle x,D_{w,j}x\right\rangle =\sum^{m}_{j=1}\frac{e_{w,j}\left(x\right)}{s_{w}\left(x\right)}e_{w,j}\left(x\right).
\]
This is the stated formula.
\end{proof}
We now lower bound the right-hand side in a uniform way.
\begin{lem}
\label{lem:f-3}For every $w$, 
\[
\frac{\sum^{m}_{j=1}e_{w,j}\left(x\right)^{2}}{s_{w}\left(x\right)}\ge\frac{1}{m}s_{w}\left(x\right),
\]
with the convention that both sides are $0$ if $s_{w}\left(x\right)=0$.
\end{lem}

\begin{proof}
If $s_{w}\left(x\right)=0$ there is nothing to prove. Otherwise,
Cauchy-Schwarz gives 
\[
\left(\sum^{m}_{j=1}e_{w,j}\left(x\right)\right)^{2}\le m\sum^{m}_{j=1}e_{w,j}\left(x\right)^{2},
\]
hence $\sum_{j}e_{w,j}\left(x\right)^{2}\ge\frac{1}{m}s_{w}\left(x\right)^{2}$.
Divide by $s_{w}\left(x\right)$.
\end{proof}
Now we can state the extinction theorem with full quantitative control.
\begin{thm}
\label{thm:f-4}Under \prettyref{assu:f-1}, with $c:=1-\alpha/m\in\left[0,1\right)$,
one has:
\begin{enumerate}
\item For every $n\ge0$, 
\[
\mathbb{E}_{\nu_{x}}\left[M_{n}\right]\le c^{n}M_{0}.
\]
\item $M_{n}\left(\omega\right)\to0$ for $\nu_{x}$-almost every $\omega$.
Equivalently, 
\[
\left\langle x,R_{\infty}\left(\omega\right)x\right\rangle =0\quad\text{for }\nu_{x}\text{-almost every }\omega.
\]
\end{enumerate}
\end{thm}

\begin{proof}
By \prettyref{sec:4} we have the pathwise identity 
\[
M_{n}=M_{n-1}-\left\langle x,D_{n}x\right\rangle ,\qquad n\ge1.
\]
An application of Lemmas \ref{lem:f-2} and \ref{lem:f-3} yields,
on $\left\{ \omega|n-1=w\right\} $, 
\[
\mathbb{E}_{\nu_{x}}\left[M_{n}\mid\mathcal{F}_{n-1}\right]=M_{n-1}-\mathbb{E}_{\nu_{x}}\left[\left\langle x,D_{n}x\right\rangle \mid\mathcal{F}_{n-1}\right]\le M_{n-1}-\frac{1}{m}s_{w}\left(x\right).
\]
By \prettyref{assu:f-1}, $s_{w}\left(x\right)\ge\alpha\left\langle x,R_{w}x\right\rangle =\alpha M_{n-1}$.
Hence 
\[
\mathbb{E}_{\nu_{x}}\left[M_{n}\mid\mathcal{F}_{n-1}\right]\le\left(1-\frac{\alpha}{m}\right)M_{n-1}=cM_{n-1}.
\]
Iterating expectations gives $\mathbb{E}\left[M_{n}\right]\le c^{n}M_{0}$. 

Define $Z_{n}:=c^{-n}M_{n}$. The previous inequality is equivalent
to 
\[
\mathbb{E}_{\nu_{x}}\left[Z_{n}\mid\mathcal{F}_{n-1}\right]\le Z_{n-1},
\]
so $\left\{ Z_{n}\right\} $ is a nonnegative supermartingale and
therefore converges $\nu_{x}$-almost surely to a finite limit $Z_{\infty}$.
(This is a standard convergence theorem; see, e.g., \cite{MR4226142,MR3930614}.)
Since $c^{n}\to0$, we have 
\[
0\le M_{n}=c^{n}Z_{n}\to0\cdot Z_{\infty}=0
\]
$\nu_{x}$-almost surely. Finally, $R_{n}\left(\omega\right)\to R_{\infty}\left(\omega\right)$
in SOT implies $M_{n}\left(\omega\right)\to\left\langle x,R_{\infty}\left(\omega\right)x\right\rangle $,
hence that limit is $0$ almost surely.
\end{proof}
A direct consequence is the energy identity along typical branches.
\begin{cor}
Under \prettyref{assu:f-1}, for $\nu_{x}$-almost every $\omega$,
\[
\left\langle x,R_{0}x\right\rangle =\sum^{\tau\left(\omega\right)-1}_{k=1}\left\langle x,D_{k}\left(\omega\right)x\right\rangle ,
\]
with the understanding that if $\tau\left(\omega\right)=\infty$ the
sum runs over all $k\ge1$.
\end{cor}

\begin{proof}
This is \prettyref{cor:d-10} with $\left\langle x,R_{\infty}\left(\omega\right)x\right\rangle =0$
from \prettyref{thm:f-4}.
\end{proof}

\subsection{Trace-biased extinction}\label{subsec:5-2}

Assume throughout this subsection that $R_{0}$ is trace class. Then
every residual $R_{w}$ and every dissipation $D_{w,j}$ is trace
class, and the trace-biased path measure $\nu_{\mathrm{tr}}$ from
\prettyref{sec:3} is well-defined.

For $w\in\left\{ 1,\dots,m\right\} ^{\ast}$, set 
\[
e^{\left(\mathrm{tr}\right)}_{w,j}:=\mathrm{tr}\left(D_{w,j}\right),\qquad s^{\left(\mathrm{tr}\right)}_{w}:=\sum^{m}_{j=1}e^{\left(\mathrm{tr}\right)}_{w,j}.
\]
Recall that the trace-biased transition at node $w$ is 
\[
p^{\left(\mathrm{tr}\right)}_{w,j}=\frac{e^{\left(\mathrm{tr}\right)}_{w,j}}{s^{\left(\mathrm{tr}\right)}_{w}}\quad\text{if }s^{\left(\mathrm{tr}\right)}_{w}>0,\qquad p^{\left(\mathrm{tr}\right)}_{w,0}=1\quad\text{if }s^{\left(\mathrm{tr}\right)}_{w}=0,
\]
and the absorbing symbol $0$ is used thereafter.

Along $\omega\in\Omega$, define as in \prettyref{sec:4}
\[
R_{n}\left(\omega\right)=R_{\omega|n},\qquad D_{n}\left(\omega\right)=D_{\omega|n-1,\omega_{n}},\qquad\tau\left(\omega\right)=\inf\left\{ n\ge1:\ \omega_{n}=0\right\} .
\]
Let $\mathcal{F}_{n}=\sigma\left(\omega_{1},\dots,\omega_{n}\right)$
be the natural filtration, and set 
\[
T_{n}\left(\omega\right):=\mathrm{tr}\left(R_{n}\left(\omega\right)\right),\qquad n\ge0.
\]
We impose the trace analog of leakage.
\begin{assumption}[uniform trace leakage]
\label{assu:f-6} There exists $\alpha\in\left(0,m\right]$ such
that for every $w\in\left\{ 1,\dots,m\right\} ^{\ast}$, 
\[
s^{\left(\mathrm{tr}\right)}_{w}\ge\alpha\,\mathrm{tr}\left(R_{w}\right).
\]
\end{assumption}

The next lemma is the conditional expectation computation in the trace-biased
model.
\begin{lem}
\label{lem:f-7}Let $n\ge1$ and condition on the event $\left\{ \omega|n-1=w\right\} $,
where $w\in\left\{ 1,\dots,m\right\} ^{\ast}$. Then 
\[
\mathbb{E}_{\nu_{\mathrm{tr}}}\left[\mathrm{tr}\left(D_{n}\left(\omega\right)\right)\mid\mathcal{F}_{n-1}\right]=\begin{cases}
{\displaystyle \frac{\sum^{m}_{j=1}\mathrm{tr}\left(D_{w,j}\right)^{2}}{s^{\left(\mathrm{tr}\right)}_{w}},} & s^{\left(\mathrm{tr}\right)}_{w}>0,\\[10pt]
0, & s^{\left(\mathrm{tr}\right)}_{w}=0.
\end{cases}
\]
\end{lem}

\begin{proof}
If $s^{\left(\mathrm{tr}\right)}_{w}=0$, then the transition is forced
to $0$ and, by the absorbed dissipation convention, the dissipation
at that step is $0$. If $s^{\left(\mathrm{tr}\right)}_{w}>0$, then
$\omega_{n}=j\in\left\{ 1,\dots,m\right\} $ with probabilities $p^{\left(\mathrm{tr}\right)}_{w,j}=\mathrm{tr}\left(D_{w,j}\right)/s^{\left(\mathrm{tr}\right)}_{w}$,
and on $\left\{ \omega|n-1=w,\ \omega_{n}=j\right\} $ we have $D_{n}\left(\omega\right)=D_{w,j}$.
Thus 
\[
\mathbb{E}_{\nu_{\mathrm{tr}}}\left[\mathrm{tr}\left(D_{n}\right)\mid\omega|n-1=w\right]=\sum^{m}_{j=1}p^{\left(\mathrm{tr}\right)}_{w,j}\mathrm{tr}\left(D_{w,j}\right)=\sum^{m}_{j=1}\frac{\mathrm{tr}\left(D_{w,j}\right)}{s^{\left(\mathrm{tr}\right)}_{w}}\mathrm{tr}\left(D_{w,j}\right).
\]
 
\end{proof}
\begin{lem}
\label{lem:f-8}For every $w$, 
\[
\frac{\sum^{m}_{j=1}\mathrm{tr}\left(D_{w,j}\right)^{2}}{s^{\left(\mathrm{tr}\right)}_{w}}\ge\frac{1}{m}s^{\left(\mathrm{tr}\right)}_{w},
\]
with the convention that both sides are $0$ if $s^{\left(\mathrm{tr}\right)}_{w}=0$.
\end{lem}

\begin{proof}
Similar to the proof of \prettyref{lem:f-3}.
\end{proof}
Now we can state the trace extinction theorem.
\begin{thm}[trace extinction and vanishing tail]
\label{thm:f-9} Under \prettyref{assu:f-6}, with $c:=1-\alpha/m\in\left[0,1\right)$,
one has:
\begin{enumerate}
\item For every $n\ge0$, 
\[
\mathbb{E}_{\nu_{\mathrm{tr}}}\left[T_{n}\right]\le c^{n}T_{0}.
\]
\item $T_{n}\left(\omega\right)\to0$ for $\nu_{\mathrm{tr}}$-almost every
$\omega$. In particular, 
\[
\mathrm{tr}\left(R_{\infty}\left(\omega\right)\right)=0\quad\text{for }\nu_{\mathrm{tr}}\text{-almost every }\omega,
\]
and therefore 
\[
R_{\infty}\left(\omega\right)=0\quad\text{for }\nu_{\mathrm{tr}}\text{-almost every }\omega.
\]
\end{enumerate}
\end{thm}

\begin{proof}
Recall that
\[
R_{n}=R_{n-1}-D_{n},\qquad n\ge1,
\]
hence $T_{n}=T_{n-1}-\mathrm{tr}\left(D_{n}\right)$. Using Lemmas
\ref{lem:f-7}-\ref{lem:f-8}, we get 
\[
\mathbb{E}_{\nu_{\mathrm{tr}}}\left[T_{n}\mid\mathcal{F}_{n-1}\right]\le T_{n-1}-\frac{1}{m}s^{\left(\mathrm{tr}\right)}_{w}.
\]
Now \prettyref{assu:f-6} gives $s^{\left(\mathrm{tr}\right)}_{w}\ge\alpha\,\mathrm{tr}\left(R_{w}\right)=\alpha T_{n-1}$,
so 
\[
\mathbb{E}_{\nu_{\mathrm{tr}}}\left[T_{n}\mid\mathcal{F}_{n-1}\right]\le\left(1-\frac{\alpha}{m}\right)T_{n-1}=cT_{n-1}.
\]
Iteration gives part (1). 

For (2), define $U_{n}:=c^{-n}T_{n}$. Then 
\[
\mathbb{E}_{\nu_{\mathrm{tr}}}\left[U_{n}\mid\mathcal{F}_{n-1}\right]\le U_{n-1},
\]
so $\left\{ U_{n}\right\} $ is a nonnegative supermartingale and
converges $\nu_{\mathrm{tr}}$-almost surely to a finite limit. Since
$c^{n}\to0$, we have $T_{n}=c^{n}U_{n}\to0$ almost surely. 

Finally, for each $\omega$ we have $R_{n}(\omega)\xrightarrow{s}R_{\infty}(\omega)$
and $0\le R_{n}(\omega)\le R_{0}$. Since $R_{0}$ is trace class
and $0\le R_{n}(\omega)\le R_{0}$, it follows that each $R_{n}(\omega)$
and $R_{\infty}(\omega)$ is trace class. Fix $\omega$ and let $\left\{ P_{k}\right\} _{k\ge1}$
be an increasing sequence of finite-rank projections with $P_{k}\xrightarrow{s}I$.
For each $k$ we have 
\[
P_{k}R_{n}(\omega)P_{k}\xrightarrow[n\to\infty]{}P_{k}R_{\infty}(\omega)P_{k}
\]
in operator norm (since the range of $P_{k}$ is finite-dimensional),
hence 
\[
\mathrm{tr}\left(P_{k}R_{n}(\omega)P_{k}\right)\xrightarrow[n\to\infty]{}\mathrm{tr}\left(P_{k}R_{\infty}(\omega)P_{k}\right).
\]
On the other hand, 
\[
0\le\mathrm{tr}\left(P_{k}R_{n}(\omega)P_{k}\right)\le\mathrm{tr}\left(R_{n}(\omega)\right)\xrightarrow[n\to\infty]{}0,
\]
so $\mathrm{tr}\left(P_{k}R_{\infty}(\omega)P_{k}\right)=0$ for every
$k$. Since $R_{\infty}(\omega)\ge0$ is trace class, we have 
\[
\mathrm{tr}\left(R_{\infty}(\omega)\right)=\lim_{k\to\infty}\mathrm{tr}\left(P_{k}R_{\infty}(\omega)P_{k}\right)=0,
\]
and hence $R_{\infty}(\omega)=0$ (see e.g., \cite{MR2154153,MR493419}).
This holds for $\nu_{\mathrm{tr}}$-almost every $\omega$, which
proves (2).
\end{proof}
A useful corollary is the trace reconstruction identity along typical
branches.
\begin{cor}
Under \prettyref{assu:f-6}, for $\nu_{\mathrm{tr}}$-almost every
$\omega$, 
\[
\mathrm{tr}\left(R_{0}\right)=\sum^{\tau\left(\omega\right)-1}_{k=1}\mathrm{tr}\left(D_{k}\left(\omega\right)\right),
\]
with the understanding that if $\tau\left(\omega\right)=\infty$ the
sum runs over all $k\ge1$.
\end{cor}

\begin{proof}
Apply $\mathrm{tr}$ to the stopped telescoping identity and use $R_{\infty}\left(\omega\right)=0$
from \prettyref{thm:f-9}. 
\end{proof}
\begin{rem*}[Operator tail vs. quadratic-form tail]
 In Section \ref{subsec:5-1}, the argument yields $\left\langle x,R_{\infty}\left(\omega\right)x\right\rangle =0$,
$\nu_{x}$-a.s., for a fixed $x$, which does not by itself force
$R_{\infty}\left(\omega\right)=0$. By contrast, in the trace-biased
setting we obtain $R_{\infty}\left(\omega\right)=0$ $\nu_{\mathrm{tr}}$-a.s.
(since $R_{\infty}\left(\omega\right)\ge0$ is trace class and ${\rm tr}\left(R_{\infty}\left(\omega\right)\right)=0$).
Thus the trace-biased measure controls the full operator tail, not
just a single quadratic form. Accordingly, the two sections \ref{subsec:5-1}-\ref{subsec:5-2}
are parallel, but the trace-biased conclusion is strictly stronger.
\end{rem*}

\section{Radon-Nikodym densities and change of measure}\label{sec:6}

In this section we compare the two intrinsic path measures introduced
earlier: the energy-biased measures $\nu_{x}$ (depending on a starting
vector $x$) and the trace-biased measure $\nu_{\mathrm{tr}}$. The
basic observation is that, on cylinder events, these measures differ
by explicit stepwise likelihood ratios built from the corresponding
transition probabilities. We first identify the resulting Radon-Nikodym
derivative $d\nu_{x}/d\nu_{\mathrm{tr}}$ as a canonical martingale
limit. We then lift this scalar change of measure to an operator-valued
disintegration, producing an $S_{1}(H)_{+}$-valued density whose
$\nu_{\mathrm{tr}}$-integral recovers the initial trace-class datum.
\begin{notation*}
For $1\le p<\infty$ we write $S_{p}\left(H\right)$ for the Schatten
$p$-class on $H$, and $S_{p}\left(H\right)_{+}:=S_{p}\left(H\right)\cap B\left(H\right)_{+}$
for its positive cone. In particular, $S_{1}\left(H\right)$ is the
trace class and $S_{2}\left(H\right)$ is the Hilbert-Schmidt class.
\end{notation*}

\subsection{Scalar Radon-Nikodym densities}

In this subsection we show that the trace-biased measure $\nu_{\mathrm{tr}}$
dominates each energy-biased measure $\nu_{x}$, and we identify $d\nu_{x}/d\nu_{\mathrm{tr}}$
as the limit of the finite-level likelihood ratios on $\mathcal{F}_{n}$.

Assume $R_{0}$ is trace class. Fix $x\in H$ with $\left\langle x,R_{0}x\right\rangle >0$.
Recall that for each node $w\in\left\{ 1,\dots,m\right\} ^{\ast}$
we have 
\[
e_{w,j}\left(x\right):=\left\langle x,D_{w,j}x\right\rangle ,\qquad s_{w}\left(x\right):=\sum^{m}_{j=1}e_{w,j}\left(x\right),
\]
and 
\[
e^{\left(\mathrm{tr}\right)}_{w,j}:=\mathrm{tr}\left(D_{w,j}\right),\qquad s^{\left(\mathrm{tr}\right)}_{w}:=\sum^{m}_{j=1}e^{\left(\mathrm{tr}\right)}_{w,j}.
\]
The transition probabilities are 
\[
p^{\left(x\right)}_{w,j}=\frac{e_{w,j}\left(x\right)}{s_{w}\left(x\right)}\ \text{ when }s_{w}\left(x\right)>0,\qquad p^{\left(\mathrm{tr}\right)}_{w,j}=\frac{e^{\left(\mathrm{tr}\right)}_{w,j}}{s^{\left(\mathrm{tr}\right)}_{w}}\ \text{ when }s^{\left(\mathrm{tr}\right)}_{w}>0,
\]
and at dead nodes we terminate into the absorbing symbol $0$. Let
$\mathcal{F}_{n}=\sigma\left(\omega_{1},\dots,\omega_{n}\right)$
be the coordinate filtration on $\Omega$.
\begin{thm}
\label{thm:6-1}For every $x\in H$ with $\left\langle x,R_{0}x\right\rangle >0$,
one has $\nu_{x}\ll\nu_{\mathrm{tr}}$. Moreover, there exists an
$\mathcal{F}_{\infty}$-measurable function $\rho_{x}:\Omega\to\left[0,\infty\right)$
such that 
\[
\nu_{x}(E)=\int_{E}\rho_{x}\left(\omega\right)d\nu_{\mathrm{tr}}\left(\omega\right)
\]
\textup{for every Borel set $E\subset\Omega$, }and $\rho_{x}$ arises
as the almost sure limit of a canonical $\nu_{\mathrm{tr}}$-martingale.

More precisely, define on cylinders the likelihood ratio 
\[
\rho_{x,n}\left(\omega\right):=\frac{\nu_{x}\left(\left[\omega|n\right]\right)}{\nu_{\mathrm{tr}}\left(\left[\omega|n\right]\right)}
\]
whenever $\nu_{\mathrm{tr}}\left(\left[\omega|n\right]\right)>0$,
and set $\rho_{x,n}\left(\omega\right)=0$ otherwise. Then:
\begin{enumerate}
\item $\left\{ \rho_{x,n}\right\} _{n\ge0}$ is a nonnegative $\nu_{\mathrm{tr}}$-martingale
with $\mathbb{E}_{\nu_{\mathrm{tr}}}\left[\rho_{x,n}\right]=1$ for
all $n$. 
\item $\rho_{x,n}\to\rho_{x}$ $\nu_{\mathrm{tr}}$-almost surely and in
$L^{1}\left(\nu_{\mathrm{tr}}\right)$. 
\item For $\nu_{\mathrm{tr}}$-almost every $\omega$, $\rho_{x}\left(\omega\right)$
admits the multiplicative path expansion 
\[
\rho_{x}\left(\omega\right)=\prod^{\tau\left(\omega\right)-1}_{k=1}\frac{p^{\left(x\right)}_{\omega|k-1,\omega_{k}}}{p^{\left(\mathrm{tr}\right)}_{\omega|k-1,\omega_{k}}},
\]
with the convention that if $\tau\left(\omega\right)=\infty$ the
product runs over all $k\ge1$. 
\end{enumerate}
\end{thm}

\begin{proof}
\textbf{Step 1. }To see absolute continuity, it suffices to verify
$\nu_{x}\ll\nu_{\mathrm{tr}}$ on each $\mathcal{F}_{n}$, equivalently
that $\nu_{\mathrm{tr}}\left(\left[u\right]\right)=0\Rightarrow\nu_{x}\left(\left[u\right]\right)=0$
for every length-$n$ cylinder $\left[u\right]$. Fix a cylinder $[u]$
with $u\in\mathcal{A}^{\ast}$. If $u$ contains a $0$, then both
measures force absorption after the first $0$, and the cylinder probability
is determined by the transitions up to that time; so we reduce to
the case where $u\in\left\{ 1,\dots,m\right\} ^{n}$.

Write $u=u_{1}\cdots u_{n}$. If $\nu_{\mathrm{tr}}\left(\left[u\right]\right)=0$,
then at some step $k$ the trace transition along $u_{k}$ vanishes:
\[
p^{\left(\mathrm{tr}\right)}_{u_{1}\cdots u_{k-1},u_{k}}=0.
\]
If $s^{\left(\mathrm{tr}\right)}_{u_{1}\cdots u_{k-1}}=0$, then the
trace process would terminate at that node (probability $1$ on the
symbol $0$), hence any genuine child $u_{k}\in\left\{ 1,\dots,m\right\} $
has probability $0$, as required. Otherwise $s^{\left(\mathrm{tr}\right)}_{u_{1}\cdots u_{k-1}}>0$
and 
\[
p^{\left(\mathrm{tr}\right)}_{u_{1}\cdots u_{k-1},u_{k}}=\frac{\mathrm{tr}\left(D_{u_{1}\cdots u_{k-1},u_{k}}\right)}{s^{\left(\mathrm{tr}\right)}_{u_{1}\cdots u_{k-1}}}=0,
\]
so $\mathrm{tr}\left(D_{u_{1}\cdots u_{k-1},u_{k}}\right)=0$. Since
$D_{u_{1}\cdots u_{k-1},u_{k}}\ge0$ is trace class, $\mathrm{tr}\left(D\right)=0$
implies $D=0$. Hence 
\[
e_{u_{1}\cdots u_{k-1},u_{k}}\left(x\right)=\left\langle x,D_{u_{1}\cdots u_{k-1},u_{k}}x\right\rangle =0,
\]
and therefore the energy transition $p^{\left(x\right)}_{u_{1}\cdots u_{k-1},u_{k}}$
is also $0$. This forces $\nu_{x}\left(\left[u\right]\right)=0$.
Thus $\nu_{x}\ll\nu_{\mathrm{tr}}$.

\textbf{Step 2.}\textit{ }We show that, for each $n$, $\rho_{x,n}$
is $\mathcal{F}_{n}$-measurable and nonnegative. Let $u\in\mathcal{A}^{n}$
be a length-$n$ word, and assume $\nu_{\mathrm{tr}}\left(\left[u\right]\right)>0$.
For any $a\in\mathcal{A}$, 
\[
\nu_{\mathrm{tr}}\left(\left[ua\right]\right)=\nu_{\mathrm{tr}}\left(\left[u\right]\right)p^{\left(\mathrm{tr}\right)}_{u,a},\qquad\nu_{x}\left(\left[ua\right]\right)=\nu_{x}\left(\left[u\right]\right)p^{\left(x\right)}_{u,a},
\]
by the cylinder recursion from \prettyref{sec:3}. Therefore, on the
cylinder $\left[ua\right]$, 
\[
\rho_{x,n+1}=\frac{\nu_{x}\left(\left[ua\right]\right)}{\nu_{\mathrm{tr}}\left(\left[ua\right]\right)}=\frac{\nu_{x}\left(\left[u\right]\right)}{\nu_{\mathrm{tr}}\left(\left[u\right]\right)}\cdot\frac{p^{\left(x\right)}_{u,a}}{p^{\left(\mathrm{tr}\right)}_{u,a}}=\rho_{x,n}\cdot\frac{p^{\left(x\right)}_{u,a}}{p^{\left(\mathrm{tr}\right)}_{u,a}},
\]
with the understanding that if $p^{\left(\mathrm{tr}\right)}_{u,a}=0$
then also $p^{\left(x\right)}_{u,a}=0$ by Step 1, and we take the
ratio to be $0$. Now compute conditional expectation given $\mathcal{F}_{n}$.
On $\left\{ \omega|n=u\right\} $, 
\[
\mathbb{E}_{\nu_{\mathrm{tr}}}\left[\rho_{x,n+1}\mid\mathcal{F}_{n}\right]=\sum_{a\in\mathcal{A}}\rho_{x,n}\cdot\frac{p^{\left(x\right)}_{u,a}}{p^{\left(\mathrm{tr}\right)}_{u,a}}\cdot p^{\left(\mathrm{tr}\right)}_{u,a}=\rho_{x,n}\sum_{a\in\mathcal{A}}p^{\left(x\right)}_{u,a}=\rho_{x,n}.
\]
Thus $\left\{ \rho_{x,n}\right\} $ is a $\nu_{\mathrm{tr}}$-martingale.
(This is the usual likelihood-ratio (change-of-measure) martingale
associated with the finite-level Radon-Nikodym derivatives; see \cite{MR4226142,MR1155402}.)
Since $\rho_{x,0}=1$, taking expectations yields $\mathbb{E}_{\nu_{\mathrm{tr}}}\left[\rho_{x,n}\right]=1$
for all $n$.

\textbf{Step 3.} By the martingale convergence theorem, $\rho_{x,n}$
converges $\nu_{\mathrm{tr}}$-almost surely to a limit $\rho_{x}\in L^{1}\left(\nu_{\mathrm{tr}}\right)$.
Moreover, since $\rho_{x,n}=d\nu_{x}|_{\mathcal{F}_{n}}/d\nu_{\mathrm{tr}}|_{\mathcal{F}_{n}}$
is a density martingale, $\{\rho_{x,n}\}$ is uniformly integrable,
hence $\rho_{x,n}\to\rho_{x}$ in $L^{1}\left(\nu_{\mathrm{tr}}\right)$
(see \cite{MR4226142,MR2267655}). Finally, for any cylinder set $[u]\in\mathcal{F}_{n}$,
\begin{align*}
\int_{\left[u\right]}\rho_{x,n}\,d\nu_{\mathrm{tr}} & =\sum_{v\in\mathcal{A}^{n},\:\left[v\right]\subset\left[u\right]}\int_{\left[v\right]}\rho_{x,n}\,d\nu_{\mathrm{tr}}\\
 & =\sum_{v\in\mathcal{A}^{n},\:\left[v\right]\subset\left[u\right]}\left(\frac{\nu_{x}\left(\left[v\right]\right)}{\nu_{\mathrm{tr}}\left(\left[v\right]\right)}\right)\nu_{\mathrm{tr}}\left(\left[v\right]\right)\\
 & =\sum_{\substack{v\in\mathcal{A}^{n},\:u\prec v}
}\nu_{x}\left(\left[v\right]\right)=\nu_{x}\left(\left[u\right]\right).
\end{align*}
Letting $n\to\infty$ and using $L^{1}$ convergence gives $\nu_{x}\left(\left[u\right]\right)=\int_{\left[u\right]}\rho_{x}\,d\nu_{\mathrm{tr}}$
for every cylinder $\left[u\right]$. Since cylinders generate the
Borel $\sigma$-algebra and both sides define measures, the identity
extends to all Borel sets $E$.

\textbf{Step 4.} The cylinder recursion in Step 2 iterated along $\omega$
yields 
\[
\rho_{x,n}\left(\omega\right)=\prod^{n}_{k=1}\frac{p^{\left(x\right)}_{\omega|k-1,\omega_{k}}}{p^{\left(\mathrm{tr}\right)}_{\omega|k-1,\omega_{k}}},
\]
and by absorption all factors are $1$ once $\omega_{k}=0$. Hence
the product truncates at $\tau\left(\omega\right)-1$, giving the
stated formula for $\rho_{x}$ as the almost sure limit.
\end{proof}
The measures $\nu_{x}$ and $\nu_{\mathrm{tr}}$ are both intrinsic
to the same energy tree, but they arise from different bias functionals
and are therefore a priori unrelated. \prettyref{thm:6-1} shows that
once $R_{0}$ is trace class, $\nu_{\mathrm{tr}}$ dominates the entire
family $\left\{ \nu_{x}\right\} $, so one may treat $\nu_{\mathrm{tr}}$
as a reference measure and express $\nu_{x}$ via the density $\rho_{x}=d\nu_{x}/d\nu_{\mathrm{tr}}$.
This density is not chosen separately; it is the canonical martingale
limit of the cylinder likelihood ratios determined by the transition
structure. Here, ``canonical'' means that for each $n$ the restriction
$\nu_{x}|_{\mathcal{F}_{n}}$ has a unique $\mathcal{F}_{n}$-measurable
density $d\nu_{x}|_{\mathcal{F}_{n}}/d\nu_{\mathrm{tr}}|_{\mathcal{F}_{n}}$,
and $\rho_{x}$ is the $\nu_{\mathrm{tr}}$-a.s. limit of these densities.
This change-of-measure step is the starting point for the operator-valued
Radon-Nikodym disintegration below.

\subsection{Operator-valued densities}\label{subsec:6-2}

We now use $\nu_{\mathrm{tr}}$ to integrate the pathwise dissipation
in trace class. The resulting random operator $\Sigma\left(\omega\right)\in S_{1}\left(H\right)_{+}$
yields an $S_{1}\left(H\right)_{+}$-valued measure $\mathsf{M}$
that is absolutely continuous with respect to $\nu_{\mathrm{tr}}$,
with total mass tied to $R_{0}$ and the tail $R_{\infty}\left(\omega\right)$.

Assume $R_{0}$ is trace class, and let $\nu_{\mathrm{tr}}$ be the
trace-biased path measure from \prettyref{sec:3}. Along $\omega\in\Omega$
we have the stopped dissipation process $D_{n}\left(\omega\right)$
and residuals $R_{n}\left(\omega\right)$ from \prettyref{sec:4},
with termination time $\tau\left(\omega\right)$.

Define the pathwise total dissipation 
\[
\Sigma\left(\omega\right):=\sum^{\tau\left(\omega\right)-1}_{n=1}D_{n}\left(\omega\right),
\]
where the sum is understood as the limit of its increasing partial
sums in the trace norm.
\begin{thm}
\label{thm:6-2}With the above assumptions and notation:
\begin{enumerate}
\item For $\nu_{\mathrm{tr}}$-almost every $\omega$, the series defining
$\Sigma\left(\omega\right)$ converges in trace norm to an element
of $S_{1}\left(H\right)_{+}$, and one has the identity 
\[
\Sigma\left(\omega\right)=R_{0}-R_{\infty}\left(\omega\right)\quad\text{in }S_{1}\left(H\right),
\]
where $R_{\infty}\left(\omega\right)={\rm s\text{-}lim}_{n\to\infty}R_{n}\left(\omega\right)$
is the tail residual from \prettyref{sec:4}. 
\item The set function 
\[
\mathsf{M}\left(E\right):=\sum^{\infty}_{n=1}\int_{E}D_{n}\left(\omega\right)d\nu_{\mathrm{tr}}\left(\omega\right),\qquad E\subset\mathcal{B}\left(\Omega\right)\ \text{Borel},
\]
is well-defined as an $S_{1}\left(H\right)_{+}$-valued, countably
additive measure. Moreover, $\mathsf{M}\ll\nu_{\mathrm{tr}}$ and
\[
\mathsf{M}\left(E\right)=\int_{E}\Sigma\left(\omega\right)d\nu_{\mathrm{tr}}\left(\omega\right)\quad\text{for all Borel }E\subset\Omega.
\]
\item The total mass satisfies 
\[
\mathsf{M}\left(\Omega\right)=R_{0}-\int_{\Omega}R_{\infty}\left(\omega\right)d\nu_{\mathrm{tr}}\left(\omega\right),
\]
and, in particular, 
\[
\mathrm{tr}\left(\mathsf{M}\left(\Omega\right)\right)=\mathrm{tr}\left(R_{0}\right)-\int_{\Omega}\mathrm{tr}\left(R_{\infty}\left(\omega\right)\right)d\nu_{\mathrm{tr}}\left(\omega\right).
\]
\item Under the trace extinction hypothesis of \prettyref{thm:f-9}, $\Sigma\left(\omega\right)=R_{0}$
for $\nu_{\mathrm{tr}}$-almost every $\omega$, hence $\mathsf{M}\left(E\right)=\nu_{\mathrm{tr}}\left(E\right)R_{0}$. 
\end{enumerate}
\end{thm}

\begin{proof}
For each $\omega$, the telescoping identity from \prettyref{sec:4}
gives, for every $n\ge1$, 
\[
R_{0}-R_{n}\left(\omega\right)=\sum^{n}_{k=1}D_{k}\left(\omega\right)\qquad\text{in }B\left(H\right),
\]
and the right-hand side is an increasing sequence in $B\left(H\right)_{+}$.
Since $R_{0}$ is trace class and $0\le R_{n}\left(\omega\right)\le R_{0}$,
each $R_{n}\left(\omega\right)$ is trace class and 
\[
\mathrm{tr}\left(R_{0}-R_{n}\left(\omega\right)\right)=\sum^{n}_{k=1}\mathrm{tr}\left(D_{k}\left(\omega\right)\right)\le\mathrm{tr}\left(R_{0}\right).
\]
In particular, for each $\omega$ the scalar sequence $\mathrm{tr}\left(R_{n}\left(\omega\right)\right)$
is decreasing and bounded below by $0$, hence convergent. Now for
$n>m$, 
\[
\left\Vert R_{m}\left(\omega\right)-R_{n}\left(\omega\right)\right\Vert _{1}=\mathrm{tr}\left(R_{m}\left(\omega\right)-R_{n}\left(\omega\right)\right)=\mathrm{tr}\left(R_{m}\left(\omega\right)\right)-\mathrm{tr}\left(R_{n}\left(\omega\right)\right),
\]
since $R_{m}\left(\omega\right)\ge R_{n}\left(\omega\right)\ge0$.
Therefore $\left\{ R_{n}\left(\omega\right)\right\} $ is Cauchy in
trace norm for every $\omega$, and there exists $\widetilde{R}_{\infty}\left(\omega\right)\in S_{1}\left(H\right)_{+}$
such that 
\[
R_{n}(\omega)\to\widetilde{R}_{\infty}(\omega)\quad\text{in }\left\Vert \cdot\right\Vert _{1}.
\]
Trace-norm convergence implies strong convergence, so the trace-norm
limit must coincide with the SOT limit from \prettyref{sec:4}. Hence
\[
\widetilde{R}_{\infty}\left(\omega\right)=R_{\infty}\left(\omega\right),\qquad R_{n}\left(\omega\right)\to R_{\infty}\left(\omega\right)\ \text{in }\left\Vert \cdot\right\Vert _{1}.
\]
This also forces $\mathrm{tr}\left(R_{n}\left(\omega\right)\right)\to\mathrm{tr}\left(R_{\infty}\left(\omega\right)\right)$.

From the finite telescoping identity, 
\[
\sum^{n}_{k=1}D_{k}\left(\omega\right)=R_{0}-R_{n}\left(\omega\right).
\]
Letting $n\to\infty$ and using trace-norm convergence $R_{n}\left(\omega\right)\to R_{\infty}\left(\omega\right)$
yields 
\[
\sum^{\infty}_{k=1}D_{k}\left(\omega\right)=R_{0}-R_{\infty}\left(\omega\right)\quad\text{in }\left\Vert \cdot\right\Vert _{1}.
\]
By the stopping-time convention in \prettyref{sec:4}, $D_{k}\left(\omega\right)=0$
for all $k\ge\tau\left(\omega\right)$, so the infinite sum truncates
and equals $\sum^{\tau\left(\omega\right)-1}_{k=1}D_{k}\left(\omega\right)$.
This proves (1), with $\Sigma\left(\omega\right)=R_{0}-R_{\infty}\left(\omega\right)$
in $S_{1}\left(H\right)$.

For each fixed $n$, the map $E\mapsto\int_{E}D_{n}\left(\omega\right)d\nu_{\mathrm{tr}}\left(\omega\right)$
is an $S_{1}\left(H\right)$-valued countably additive measure (Bochner
integral), because $D_{n}\left(\omega\right)$ is strongly measurable
and 
\[
\int_{\Omega}\left\Vert D_{n}\left(\omega\right)\right\Vert _{1}d\nu_{\mathrm{tr}}\left(\omega\right)=\int_{\Omega}\mathrm{tr}\left(D_{n}\left(\omega\right)\right)d\nu_{\mathrm{tr}}\left(\omega\right)\le\mathrm{tr}\left(R_{0}\right)<\infty.
\]
(For background on Bochner integration in Banach spaces and the induced
vector measures, see \cite{MR453964,MR2267655}.)

Moreover, for disjoint Borel sets $E_{i}$, 
\[
\sum_{i}\int_{E_{i}}D_{n}\,d\nu_{\mathrm{tr}}=\int_{\cup_{i}E_{i}}D_{n}\,d\nu_{\mathrm{tr}}\quad\text{in }\left\Vert \cdot\right\Vert _{1}.
\]
Now define 
\[
\mathsf{M}\left(E\right):=\sum^{\infty}_{n=1}\int_{E}D_{n}\,d\nu_{\mathrm{tr}}.
\]
This series converges in $\left\Vert \cdot\right\Vert _{1}$ because
for $N>M$, 
\[
\left\Vert \sum^{N}_{n=M+1}\int_{E}D_{n}\,d\nu_{\mathrm{tr}}\right\Vert _{1}\le\sum^{N}_{n=M+1}\int_{E}\left\Vert D_{n}\right\Vert _{1}\,d\nu_{\mathrm{tr}}=\sum^{N}_{n=M+1}\int_{E}\mathrm{tr}\left(D_{n}\right)d\nu_{\mathrm{tr}},
\]
and the right-hand side is bounded by $\sum_{n\ge1}\int_{\Omega}\mathrm{tr}\left(D_{n}\right)d\nu_{\mathrm{tr}}\le\mathrm{tr}\left(R_{0}\right)$,
using telescoping in expectation: 
\[
\sum^{N}_{n=1}\int_{\Omega}\mathrm{tr}\left(D_{n}\right)d\nu_{\mathrm{tr}}=\int_{\Omega}\mathrm{tr}\left(R_{0}-R_{N}\left(\omega\right)\right)d\nu_{\mathrm{tr}}\le\mathrm{tr}\left(R_{0}\right).
\]
Thus $\mathsf{M}(E)\in S_{1}(H)_{+}$ is well-defined. Countable additivity
follows by exchanging the sum over $n$ with the countable additivity
in $E$ (justified by the monotone bound above in trace norm): for
disjoint $E_{i}$, 
\begin{align*}
\mathsf{M}\left(\cup_{i}E_{i}\right) & =\sum^{\infty}_{n=1}\int_{\cup_{i}E_{i}}D_{n}\,d\nu_{\mathrm{tr}}\\
 & =\sum^{\infty}_{n=1}\sum_{i}\int_{E_{i}}D_{n}\,d\nu_{\mathrm{tr}}=\sum_{i}\sum^{\infty}_{n=1}\int_{E_{i}}D_{n}\,d\nu_{\mathrm{tr}}=\sum_{i}\mathsf{M}(E_{i}),
\end{align*}
with convergence in $\left\Vert \cdot\right\Vert _{1}$.

From the above discussion, we have $\Sigma\left(\omega\right)=\sum_{n\ge1}D_{n}\left(\omega\right)$
in trace norm for $\nu_{\mathrm{tr}}$-almost every $\omega$, and
$\left\Vert \Sigma\left(\omega\right)\right\Vert _{1}=\mathrm{tr}\left(\Sigma\left(\omega\right)\right)\le\mathrm{tr}\left(R_{0}\right)$,
so $\Sigma\in L^{1}\left(\nu_{\mathrm{tr}};S_{1}\left(H\right)\right)$.
For each Borel set $E$, dominated convergence in $S_{1}\left(H\right)$
(using the nonnegative trace bound) yields 
\[
\int_{E}\Sigma\left(\omega\right)d\nu_{\mathrm{tr}}\left(\omega\right)=\int_{E}\sum^{\infty}_{n=1}D_{n}\left(\omega\right)d\nu_{\mathrm{tr}}\left(\omega\right)=\sum^{\infty}_{n=1}\int_{E}D_{n}\left(\omega\right)d\nu_{\mathrm{tr}}\left(\omega\right)=\mathsf{M}\left(E\right),
\]
which proves (2) and shows $\mathsf{M}\ll\nu_{\mathrm{tr}}$ with
Radon-Nikodym derivative $\Sigma$. 

Finally, taking $E=\Omega$ and using $\Sigma=R_{0}-R_{\infty}$ gives
(3). Under trace extinction (\prettyref{thm:f-9}), $R_{\infty}\left(\omega\right)=0$
almost surely, so $\Sigma\left(\omega\right)=R_{0}$ almost surely
and $\mathsf{M}\left(E\right)=\nu_{\mathrm{tr}}\left(E\right)R_{0}$. 
\end{proof}

\section{Boundary and disintegration}\label{sec:7}

Assume throughout that $R_{0}\in S_{1}\left(H\right)_{+}$, so the
trace-biased path measure $\nu_{\mathrm{tr}}$ is defined. By \prettyref{sec:4},
along each $\omega\in\Omega$ we have a decreasing residual process
$R_{n}\left(\omega\right)\downarrow R_{\infty}\left(\omega\right)$
in SOT, and by \prettyref{thm:6-2} this convergence holds in trace
norm. We therefore regard 
\[
R_{\infty}:\Omega\to S_{1}\left(H\right)_{+}
\]
as the canonical ``boundary variable'' attached to the WR energy
tree, and we write 
\[
\Sigma\left(\omega\right):=R_{0}-R_{\infty}\left(\omega\right)\in S_{1}\left(H\right)_{+}.
\]
By \prettyref{thm:6-2}, $\Sigma\left(\omega\right)=\sum^{\tau\left(\omega\right)-1}_{n=1}D_{n}\left(\omega\right)$
with convergence in $\left\Vert \cdot\right\Vert _{1}$, and $\left\Vert \Sigma\left(\omega\right)\right\Vert _{1}\le\mathrm{tr}\left(R_{0}\right)$.
\begin{defn}
The WR boundary $\sigma$-field is 
\[
\mathcal{B}_{\mathrm{WR}}:=\sigma\left(R_{\infty}\right)=\sigma\left(\Sigma\right),
\]
where the equality holds since $\Sigma=R_{0}-R_{\infty}$.

We denote by 
\[
\mu_{\mathrm{tr}}:=\left(R_{\infty}\right)_{\#}\nu_{\mathrm{tr}}
\]
the distribution of $R_{\infty}$ under $\nu_{\mathrm{tr}}$, i.e.,
\[
\mu_{\mathrm{tr}}\left(E\right)=\nu_{\mathrm{tr}}\left(\left\{ \omega\in\Omega:R_{\infty}\left(\omega\right)\in E\right\} \right)
\]
for all Borel $E\subset S_{1}\left(H\right)_{+}$. Since $\Omega$
is a compact metric space and $S_{1}\left(H\right)$ is separable,
$S_{1}\left(H\right)_{+}$ is a Polish space in the trace norm topology,
and $R_{\infty}$ is Borel measurable as an $S_{1}\left(H\right)$-valued
limit of $\mathcal{F}_{n}$-measurable simple functions.
\end{defn}

The first boundary theorem is the existence of a canonical disintegration
of $\nu_{\mathrm{tr}}$ over the boundary variable $R_{\infty}$.
\begin{thm}
\label{thm:7-2}There exists a Borel kernel 
\[
T\mapsto\nu^{T}_{\mathrm{tr}}
\]
from $S_{1}\left(H\right)_{+}$ into probability measures on $\Omega$
such that:
\begin{enumerate}
\item For every Borel set $E\subset\Omega$, the map $T\mapsto\nu^{T}_{\mathrm{tr}}\left(E\right)$
is Borel, and 
\[
\nu_{\mathrm{tr}}\left(E\right)=\int_{S_{1}\left(H\right)_{+}}\nu^{T}_{\mathrm{tr}}\left(E\right)d\mu_{\mathrm{tr}}\left(T\right).
\]
\item For $\mu_{\mathrm{tr}}$-almost every $T$, the measure $\nu^{T}_{\mathrm{tr}}$
is supported on the fiber 
\[
\left\{ \omega:R_{\infty}\left(\omega\right)=T\right\} 
\]
in the sense that 
\[
\nu^{T}_{\mathrm{tr}}\left(\left\{ \omega:\ R_{\infty}\left(\omega\right)=T\right\} \right)=1.
\]
\item For every integrable $f\in L^{1}\left(\nu_{\mathrm{tr}}\right)$,
the conditional expectation onto $\mathcal{B}_{\mathrm{WR}}$ is given
by 
\[
\mathbb{E}_{\nu_{\mathrm{tr}}}\left[f\mid\mathcal{B}_{\mathrm{WR}}\right]\left(\omega\right)=\int_{\Omega}f\left(\eta\right)d\nu^{R_{\infty}\left(\omega\right)}_{\mathrm{tr}}\left(\eta\right)\quad\text{for }\nu_{\mathrm{tr}}\text{-almost every }\omega.
\]
\end{enumerate}
\end{thm}

\begin{proof}
Since $\Omega$ is Polish and $S_{1}\left(H\right)_{+}$ is Polish,
the map $R_{\infty}:\Omega\to S_{1}\left(H\right)_{+}$ admits a regular
conditional distribution under $\nu_{\mathrm{tr}}$. Concretely, there
exists a Borel kernel $T\mapsto\nu^{T}_{\mathrm{tr}}$ such that for
every bounded Borel $g$ on $\Omega$ and every bounded Borel $\varphi$
on $S_{1}\left(H\right)_{+}$, 
\[
\int_{\Omega}g\left(\omega\right)\varphi\left(R_{\infty}\left(\omega\right)\right)\,d\nu_{\mathrm{tr}}\left(\omega\right)=\int_{S_{1}\left(H\right)_{+}}\left(\int_{\Omega}g\left(\omega\right)\,d\nu^{T}_{\mathrm{tr}}\left(\omega\right)\right)\varphi\left(T\right)\,d\mu_{\mathrm{tr}}\left(T\right).
\]
Taking $g=\mathbf{1}_{E}$ yields (1). (See, e.g., \cite{MR4226142,MR2169627}
for standard treatments of regular conditional distributions/disintegration
on Polish spaces.)

The support property (2) is a standard consequence of regular conditional
distributions: it follows by applying the preceding identity with
$g=\mathbf{1}_{\left\{ R_{\infty}\in A\right\} }$ and varying $A$,
and then using a generating $\pi$-system of Borel sets in $S_{1}\left(H\right)_{+}$. 

Finally, (3) follows by taking $g=f$ and $\varphi$ arbitrary, which
identifies the right-hand side as a version of the conditional expectation
onto $\sigma\left(R_{\infty}\right)=\mathcal{B}_{\mathrm{WR}}$.
\end{proof}
\begin{rem*}
We will use $\left\{ \nu^{T}_{\mathrm{tr}}\right\} $ repeatedly to
express $\mathcal{B}_{\mathrm{WR}}$-conditional expectations and
to convert scalar and operator-valued quantities on $\Omega$ into
boundary integrals against $\mu_{\mathrm{tr}}$.
\end{rem*}
The second boundary theorem explains how the entire family of intrinsic
measures $\nu_{x}$ changes when viewed from the boundary. This uses
the Radon-Nikodym derivative $\rho_{x}=d\nu_{x}/d\nu_{\mathrm{tr}}$
constructed in \prettyref{sec:6}.

Fix $x\in H$ with $\left\langle x,R_{0}x\right\rangle >0$, and let
$\nu_{x}$ be the energy-biased path measure. By \prettyref{thm:6-1},
$\nu_{x}\ll\nu_{\mathrm{tr}}$ and there exists $\rho_{x}\in L^{1}\left(\nu_{\mathrm{tr}}\right)$
with $\rho_{x}\ge0$, $\int\rho_{x}\,d\nu_{\mathrm{tr}}=1$, and 
\[
\nu_{x}\left(E\right)=\int_{E}\rho_{x}\left(\omega\right)d\nu_{\mathrm{tr}}\left(\omega\right)\quad\text{for all Borel }E\subset\Omega.
\]
Define also the boundary pushforward 
\[
\mu_{x}:=\left(R_{\infty}\right)_{\#}\nu_{x}.
\]

\begin{thm}
\label{thm:7-3}For each $x\in H$ with $\left\langle x,R_{0}x\right\rangle >0$,
one has $\mu_{x}\ll\mu_{\mathrm{tr}}$. Moreover, there exists a Borel
function $h_{x}:S_{1}\left(H\right)_{+}\to\left[0,\infty\right)$
such that:
\begin{enumerate}
\item $h_{x}\in L^{1}\left(\mu_{\mathrm{tr}}\right)$ and $\int h_{x}\,d\mu_{\mathrm{tr}}=1$. 
\item $h_{x}\left(R_{\infty}\left(\omega\right)\right)$ is a version of
$\mathbb{E}_{\nu_{\mathrm{tr}}}\left[\rho_{x}\mid\mathcal{B}_{\mathrm{WR}}\right]$,
that is, 
\[
\mathbb{E}_{\nu_{\mathrm{tr}}}\left[\rho_{x}\mid\mathcal{B}_{\mathrm{WR}}\right]\left(\omega\right)=h_{x}\left(R_{\infty}\left(\omega\right)\right)\quad\text{for }\nu_{\mathrm{tr}}\text{-almost every }\omega.
\]
\item $d\mu_{x}=h_{x}\,d\mu_{\mathrm{tr}}$, i.e. 
\[
\mu_{x}\left(A\right)=\int_{A}h_{x}\left(T\right)d\mu_{\mathrm{tr}}\left(T\right)\quad\text{for all Borel }A\subset S_{1}\left(H\right)_{+}.
\]
\item For every Borel set $E\subset\Omega$, 
\[
\nu_{x}\left(E\right)=\int_{S_{1}\left(H\right)_{+}}h_{x}\left(T\right)\nu^{T}_{\mathrm{tr}}\left(E\right)d\mu_{\mathrm{tr}}\left(T\right),
\]
where $T\mapsto\nu^{T}_{\mathrm{tr}}$ is the disintegration kernel
from \prettyref{thm:7-2}. 
\end{enumerate}
\end{thm}

\begin{proof}
Since $\rho_{x}\in L^{1}\left(\nu_{\mathrm{tr}}\right)$, the conditional
expectation $\mathbb{E}_{\nu_{\mathrm{tr}}}\left[\rho_{x}\mid\mathcal{B}_{\mathrm{WR}}\right]$
exists and is $\mathcal{B}_{\mathrm{WR}}$-measurable. Because $\mathcal{B}_{\mathrm{WR}}=\sigma\left(R_{\infty}\right)$,
there is a Borel $h_{x}$ such that 
\[
\mathbb{E}_{\nu_{\mathrm{tr}}}\left[\rho_{x}\mid\mathcal{B}_{\mathrm{WR}}\right]\left(\omega\right)=h_{x}\left(R_{\infty}\left(\omega\right)\right)
\]
for $\nu_{\mathrm{tr}}$-almost every $\omega$, proving (2). (This
is an instance of the Doob-Dynkin measurable factorization lemma;
see \cite{MR4226142,MR1932358}.)

Integrating both sides gives 
\[
\int h_{x}\,d\mu_{\mathrm{tr}}=\int h_{x}\left(R_{\infty}\right)\,d\nu_{\mathrm{tr}}=\int\mathbb{E}_{\nu_{\mathrm{tr}}}\left[\rho_{x}\mid\mathcal{B}_{\mathrm{WR}}\right]d\nu_{\mathrm{tr}}=\int\rho_{x}\,d\nu_{\mathrm{tr}}=1,
\]
hence (1). 

For (3), let $A\subset S_{1}\left(H\right)_{+}$ be Borel. Then 
\begin{align*}
\mu_{x}\left(A\right) & =\nu_{x}\left(R_{\infty}\in A\right)=\int_{\left\{ R_{\infty}\in A\right\} }\rho_{x}\,d\nu_{\mathrm{tr}}\\
 & =\int_{\left\{ R_{\infty}\in A\right\} }\mathbb{E}_{\nu_{\mathrm{tr}}}\left[\rho_{x}\mid\mathcal{B}_{\mathrm{WR}}\right]d\nu_{\mathrm{tr}}\\
 & =\int_{\left\{ R_{\infty}\in A\right\} }h_{x}\left(R_{\infty}\right)d\nu_{\mathrm{tr}}=\int_{A}h_{x}\left(T\right)d\mu_{\mathrm{tr}}\left(T\right),
\end{align*}
which is exactly $\mu_{x}=h_{x}\,\mu_{\mathrm{tr}}$, and in particular
$\mu_{x}\ll\mu_{\mathrm{tr}}$.

For (4), fix a Borel set $E\subset\Omega$. Using $\nu_{x}=\rho_{x}\,\nu_{\mathrm{tr}}$
and then conditioning onto $\mathcal{B}_{\mathrm{WR}}$ gives 
\[
\nu_{x}\left(E\right)=\int\mathbf{1}_{E}\rho_{x}\,d\nu_{\mathrm{tr}}=\int\mathbf{1}_{E}\,\mathbb{E}_{\nu_{\mathrm{tr}}}\left[\rho_{x}\mid\mathcal{B}_{\mathrm{WR}}\right]d\nu_{\mathrm{tr}}=\int\mathbf{1}_{E}\,h_{x}\left(R_{\infty}\right)d\nu_{\mathrm{tr}}.
\]
Apply \prettyref{thm:7-2} (1) with the integrand $\mathbf{1}_{E}\,h_{x}\left(R_{\infty}\right)$,
noting that $h_{x}\left(R_{\infty}\right)$ is constant on each fiber
of $R_{\infty}$, to obtain 
\[
\nu_{x}\left(E\right)=\int_{S_{1}\left(H\right)_{+}}h_{x}\left(T\right)\,\nu^{T}_{\mathrm{tr}}\left(E\right)d\mu_{\mathrm{tr}}\left(T\right).
\]
\end{proof}
At this point, the usual ``boundary triviality'' criterion becomes
a corollary rather than a headline statement.
\begin{cor}[trivial boundary]
The following are equivalent:
\begin{enumerate}
\item $\mathcal{B}_{\mathrm{WR}}$ is trivial modulo $\nu_{\mathrm{tr}}$-null
sets. 
\item There exists $T\in S_{1}\left(H\right)_{+}$ such that $R_{\infty}\left(\omega\right)=T$
for $\nu_{\mathrm{tr}}$-almost every $\omega$. 
\end{enumerate}
In particular, under the trace extinction conclusion of \prettyref{thm:f-9},
one has $R_{\infty}\left(\omega\right)=0$ for $\nu_{\mathrm{tr}}$-almost
every $\omega$, hence the WR boundary is trivial.

\end{cor}

\begin{proof}
If $R_{\infty}$ is almost surely constant, then $\sigma\left(R_{\infty}\right)$
is trivial.

Conversely, if $\sigma\left(R_{\infty}\right)$ is trivial, then every
bounded Borel function of $R_{\infty}$ is almost surely constant;
applying this to a separating family of continuous linear functionals
on $S_{1}\left(H\right)$ implies $R_{\infty}$ itself is almost surely
constant. The final statement is \prettyref{thm:f-9}.
\end{proof}
Finally, \prettyref{thm:6-2} becomes a boundary formula for the operator-valued
measure $\mathsf{M}$ without any additional work, since $\Sigma$
is already $\mathcal{B}_{\mathrm{WR}}$-measurable.
\begin{cor}
Let $\mathsf{M}$ be the $S_{1}\left(H\right)_{+}$-valued measure
from \prettyref{thm:6-2}, so $\mathsf{M}\left(E\right)=\int_{E}\Sigma\,d\nu_{\mathrm{tr}}$.
Then for every Borel $A\subset S_{1}\left(H\right)_{+}$, 
\[
\mathsf{M}\left(\left\{ \omega:R_{\infty}\left(\omega\right)\in A\right\} \right)=\int_{A}\left(R_{0}-T\right)d\mu_{\mathrm{tr}}\left(T\right)\quad\text{in }S_{1}\left(H\right).
\]
\end{cor}

\begin{proof}
Since $\Sigma\left(\omega\right)=R_{0}-R_{\infty}\left(\omega\right)$
and $\mu_{\mathrm{tr}}=\left(R_{\infty}\right)_{\#}\nu_{\mathrm{tr}}$,
\begin{align*}
\mathsf{M}\left(R_{\infty}\in A\right) & =\int_{\left\{ R_{\infty}\in A\right\} }\Sigma\left(\omega\right)d\nu_{\mathrm{tr}}\left(\omega\right)\\
 & =\int_{\left\{ R_{\infty}\in A\right\} }\left(R_{0}-R_{\infty}\left(\omega\right)\right)d\nu_{\mathrm{tr}}\left(\omega\right)=\int_{A}\left(R_{0}-T\right)d\mu_{\mathrm{tr}}\left(T\right).
\end{align*}
\end{proof}
\begin{rem*}
The term “boundary” here refers to the terminal trace-class random
variable $R_{\infty}$ (equivalently $\Sigma=R_{0}-R_{\infty}$) and
the associated tail $\sigma$-field $\mathcal{B}_{\mathrm{WR}}=\sigma\left(R_{\infty}\right)$.
This is a boundary in the probabilistic sense: it parameterizes the
asymptotic information carried by an infinite branch of the WR energy
tree, and \prettyref{thm:7-2} provides the corresponding disintegration
of $\nu_{\mathrm{tr}}$ over the fibers of $R_{\infty}$. We do not
introduce a geometric boundary of a space or graph here, nor do we
develop Poisson, Martin, or related potential-theoretic boundary constructions;
the boundary space is simply the operator cone $S_{1}(H)_{+}$ equipped
with the boundary law $\mu_{\mathrm{tr}}=\left(R_{\infty}\right)_{\#}\nu_{\mathrm{tr}}$.
\end{rem*}

\bibliographystyle{amsalpha}
\bibliography{ref}

\end{document}